\documentclass{amsart}
\usepackage[T1]{fontenc}
\usepackage[latin9]{inputenc}
\usepackage{amsmath,amsthm,amsfonts,amscd,amssymb,eucal,latexsym,mathrsfs}
\usepackage{stmaryrd}
\usepackage{enumerate}
\usepackage{hyperref}
\usepackage[all]{xy}

\usepackage{tikz,tikz-cd}
\usetikzlibrary{shapes.geometric}
\usepackage{tqft}

\usetikzlibrary{positioning}
\usepackage{float}
\usepackage{MnSymbol}
\usetikzlibrary{matrix}

\makeatletter
\renewcommand\part{\@startsection {part}{1}{\z@}%
                                   {-3.5ex \@plus -1ex \@minus -.2ex}%
                                   {2.3ex \@plus.2ex}%
                                   {\newpage\centering\normalfont\bfseries}}
\makeatother

\theoremstyle{plain}
\newtheorem{theorem}{Theorem}
\newtheorem*{theorem*}{Theorem}

\newtheorem{lemma}[theorem]{Lemma}

\newtheorem{proposition}[theorem]{Proposition}

\theoremstyle{definition}
\newtheorem{remark}[theorem]{Remark}

\newtheorem{definition}[theorem]{Definition}

\newcommand{\acts}{\curvearrowright}
\newcommand{\G}{\Gamma}

\newcommand{\IR}{\mathbb{R}}

\newcommand{\IZ}{\mathbb{Z}}

\newcommand{\ZI}{\mathbb{Z}}

\newcommand{\SI}{\mathbb{S}}

\DeclareMathOperator{\ran}{\mathrm{ran}}

\DeclareMathOperator{\del}{\partial}

\DeclareMathOperator{\Aut}{\mathrm{Aut}}

\DeclareMathOperator{\impl}{\Rightarrow}
\DeclareMathOperator{\inj}{\hookrightarrow}

\DeclareMathOperator{\surj}{\twoheadrightarrow}

\DeclareMathOperator{\normal}{\triangleleft}

\newcommand{\IP}{\mathbb{P}}

\newcommand{\GL}{\mathrm{GL}}

\DeclareMathOperator{\Stab}{\mathrm{Stab}}

\DeclareMathOperator{\Inn}{Inn}

\newcommand{\matr}[4]{\left(\begin{array}{cc} #1 & #2 \\ #3 & #4\end{array}\right)}

\newcommand{\sC}{\mathcal{C}}

\newcommand{\MF}{\mathbf{Flats}}
\newcommand{\MP}{\mathbf{Puzzles}}

\title[$\Aut(F_2\mathbf)$ puzzles]{$\mathbf{Aut(F}_2\mathbf)$ puzzles}

\author{Sylvain Barr\'e}
\author{Mikael Pichot}

\begin{document}

\begin{abstract}
This paper studies $\Aut(F_2)$ and groups closely related to it  from a geometric perspective.

The discussion is centered around geometric structures called ``ring puzzles'' and ``ring complexes'', and groups which can be seen as ``randomizations'' of the group $\Aut(F_2)$, in the sense of \cite{random}, and can be investigated using these structures. 
\end{abstract}

\maketitle

In a recent paper \cite{random} we discussed a class of ``random groups of higher rank''. These random groups come in various models, including in particular a ``density model'' (as in the classical construction of Gromov), and they are associated with ``deterministic data'' that are used to seed the construction.

 The deterministic data in \cite{random} was associated with nonarchimedean lattices, which were chosen in particular for their relations to rank interpolation, and the question of whether
\[
\IR^2\inj X\stackrel{\text{?}}\impl \ZI^2\inj \G
\] 
whenever $\G\acts X$ properly, $X$ is nonpositively curved and $X/\G$ compact. The so-called ``flat closing conjecture'' asserts that this has a positive answer. 

The present paper uses deterministic data associated with $\Aut(F_2)$, which produces random groups which ``resemble $\Aut(F_2)$'' in a geometrical sense.

We prove that: 

\begin{theorem}\label{T - randomization autf2}
The random group in the density models associated with $\Aut(F_2)$ satisfies the flat closing conjecture. 
\end{theorem}

More precisely,  the statement holds with overwhelming probability, at arbitrary density $\delta\in(0,1)$, and with respect to arbitrary deterministic data associated with $\Aut(F_2)$.

The corresponding result for nonarchimedean lattices was left open in \cite{random}, and the techniques  used in the present paper do not appear to shed light on the problem. 

\setcounter{tocdepth}{1}
\tableofcontents

\section{Organization of the proof}

In what sense does the random group resemble $\Aut(F_2)$?

The proof of Theorem \ref{T - randomization autf2} relies on giving an answer to this question which is precise enough to keep track of the (dis)appearance of (periodic) flat planes in the random groups.  

The text is centered around geometric structures called ``ring puzzles'' and ``ring complexes'', which are defined precisely in \textsection\ref{S - ring puzzles} and \textsection\ref{S - Nonplanar} respectively.

\bigskip

 Ring puzzles are tessellations of the Euclidean plane  with  constraints on the set of tiles  and  vertex neighbourhoods (see \textsection\ref{S - ring puzzles}). 
Associated with $\Aut(F_2)$ are ring puzzles, called $\Aut(F_2)$ puzzles, with two tiles and three types of neighbourhoods. The definition is elementary  involving simple local construction rules that are not directly   reminiscent of $\Aut(F_2)$.  Incidentally, ring puzzles give informative ``snapshots'' of the geometry of $\Aut(F_2)$ itself.

It is known  that $\Aut(F_2)$ acts properly with compact quotient on a 2-dimensional  CAT(0) space called the Brady complex, which is essentially uniquely determined, by results of Crisp and Paoluzzi \cite{CP}. 
 The space of $\Aut(F_2)$ puzzles is related to the space of flats in the Brady complex or in any ring complex of type $\Aut(F_2)$. In the case of the Brady complex itself,  the space of flat surjects onto the space of puzzles, which is an indication that $\Aut(F_2)$ plays the role of ``space of maximal rank'' in rank interpolation \cite{rd}.

\bigskip

The main objective in \textsection\ref{S - ring puzzles}, and the first step in the proof of Theorem \ref{T - randomization autf2}, is to show that:

\begin{theorem*} There is an explicit classification of all $\Aut(F_2)$ puzzles.
\end{theorem*}

 The classification  finds:
\begin{itemize}
\item 4 infinite families of puzzles
\item 9 exceptional puzzles not belonging to the families
\end{itemize}
See Theorem \ref{T - Puzzle classification} for a detailed statement. 

(Large portions of a puzzle do not determine the puzzle in general, puzzles are ``bifurcating spaces'' with an interesting (pointed) Gromov--Hausdorff topology; see Remark \ref{R - marked puzzles}.)

\bigskip

Ring complexes are defined in \textsection\ref{S - Nonplanar} as metric simplicial complexes subjects to ``ring conditions'' on the links. 
Compared to classical curvature conditions (the girth condition in dimension 2), the ring conditions serve  to control    the local flatness  and shed light on the  ``intermediate rank'' structure of $\Aut(F_2)$ rather than the curvature.  Local rules identical to that of $\Aut(F_2)$ puzzles give rise to ring complexes of type $\Aut(F_2)$.
The random group in models associated with $\Aut(F_2)$ acts on a ring complex of type $\Aut(F_2)$.

\bigskip

Theorem \ref{T - randomization autf2} is proved in \textsection\ref{S - random autf2} using the classification of puzzles. It relies on a more general result in \textsection\ref{S - Nonplanar} (Theorem \ref{Th - gromov Z2 puzzles}) establishing the flat closing conjecture under a technical assumption which is satisfied for randomizations of $\Aut(F_2)$.

\bigskip

We do not pursue the study of general ring puzzles in this paper. The definition is very much in the spirit of introducing new rank interpolation tools (compare \cite{rd} or \cite{chambers}). The ring conditions seem to give rise to an interesting set of objects to study in their own right.    

\bigskip

We will further discuss two points related to the Theorem \ref{T - randomization autf2} and its proof:

\begin{enumerate}[a)]
\item There are several significant differences between randomizations of $\Aut(F_2)$ and randomizations of nonarchimedean lattices as considered in \cite{random}.  For example, we show in \textsection\ref{S - Nonplanar} that:
\begin{theorem*}
For some choice of the deterministic data, the random group contains $\ZI^2$ with overwhelming probability at arbitrary density $\delta\in(0,1)$.
\end{theorem*}
\noindent Here the deterministic data has to be chosen suitably, it is associated the symmetric groups $S_n$. The property underlying the statement of the theorem, which is satisfied for deterministic data associated with $S_n$, is the \emph{small tori property} defined in \textsection\ref{S - Nonplanar}, see Definition \ref{D - small tori property}.

\item The above raises the question of using ring puzzles as ``blueprints'' to construct groups which resemble  $\Aut(F_2)$ in a geometrical sense.  While we can produce a wide range of random examples, finding explicit groups does not appear to be an easy task. We give an explicit construction of such a group   in \textsection\ref{S - explicit examples 2}.  
\end{enumerate}

\section{Ring puzzles}\label{S - ring puzzles}

A \emph{ring puzzle} is a tessellation of the Euclidean plane $\IR^2$ using  planar polygons  with coloured angles. The \emph{ring} at a vertex of the puzzle is the coloured circle (of length $2\pi$) of small radius around that vertex.

By \emph{ring puzzle problem} we mean the problem of classifying  all puzzles whose coloured polygons and rings belong to prescribed sets.

We  consider here  the following instance of the problem, which is related to $\Aut(F_2)$.

The two prescribed sets are:
\[
\text{ Shapes }:=\ \ \  \includegraphics[width=2.7cm,trim = 0cm .35cm 0cm 0cm]{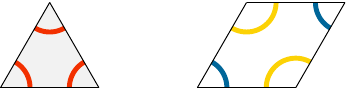}
\]
\vskip.35cm

\noindent   
(an equilateral triangle and a lozenge) and
\[
\text{ Rings }:=\   \includegraphics[width=1cm,trim = 0cm .9cm 0cm 0cm]{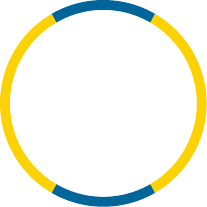}\ \includegraphics[width=1cm,trim = 0cm .9cm 0cm 0cm]{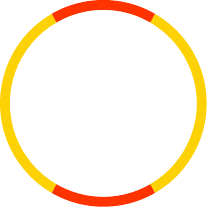}\ \includegraphics[width=1cm,trim = 0cm .9cm 0cm 0cm]{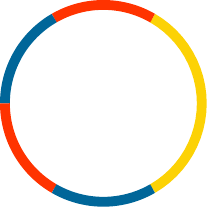}
\]
\vskip.5cm
\noindent The yellow arcs have length $2\pi/3$, the  others $\pi/3$. (For details on the relation with $\Aut(F_2)$, see  \textsection\ref{S - Nonplanar}.)  

We assume  that we have an infinite supply of coloured polygons for each shape type,  called the pieces of the puzzle. The position and orientation of the pieces  in the plane is free.

In this section we discuss the ring puzzle problem on these two sets. Every solution to the problem is called an \emph{$\Aut(F_2)$ puzzle}. 

For example, here is a permitted (bi-infinite) portion of an $\Aut(F_2)$ puzzle which we shall call the \emph{2-strip}. 

 \[
\includegraphics[width=7.5cm]{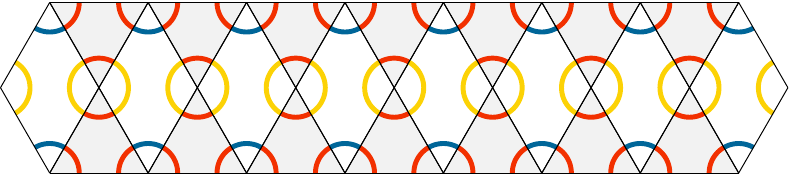}
\]

\bigskip

The main result  (Theorem \ref{T - Puzzle classification})  classifies $\Aut(F_2)$ puzzles.  This can  be achieved after a series of preliminary results on the local-global geometry of these puzzles.

Let us first isolate an important property of the two sets given above, which reflects the fact that $\Aut(F_2)$ puzzles are ``nonpositively curved'':  

\begin{definition}
A ring puzzle problem has the \emph{$\theta_0$-extension property}  if the prescribed sets satisfies the following condition: whenever a coloured segment $[0,\theta]$ of length $\theta\in [0,2\pi]$ can be isometrically embedded in distinct coloured rings  respecting the coloration, then $\theta\leq \theta_0$.  
\end{definition}

For $\Aut(F_2)$ puzzles we have that $\theta_0=\pi$, which corresponds to nonpositive curvature. 

The following consequence is straightforward:

\begin{lemma}\label{L -NPC}
Assuming the $\theta_0$-extension property,  every ring puzzle is such that  if consecutive pieces meet at a vertex contributing  to an angle $>\theta_0$, then the remaining pieces at this vertex are uniquely determined (up to colour preserving isometry). 
\end{lemma}

Let $P$ be  an $\Aut(F_2)$ puzzle. 

\begin{definition}
A $\diamond$-gallery in $P$ is a sequence of consecutive  ($=$ intersecting along an edge) lozenges. Two lozenges are connected if they belong to a $\diamond$-gallery. The  connected components of lozenges in $P$ are called the $\diamond$-components of $P$. 
\end{definition}

\begin{lemma}\label{L - components are losenges}
Let $C$ be a $\diamond$-component. The inner angle at every vertex of $\del C$ is equal to either $\pi/3$, $2\pi/3$ or $\pi$.  
\end{lemma}

\begin{proof}
The inner angle is a multiple of $\pi/3$. Assume that it is $>\pi$;  since we are on the boundary, it can only be $4\pi/3$ or $5\pi/3$. The value $4\pi/3$ would single out the first ring, however, the remaining colour (namely yellow) corresponds to a lozenge,  contradicting the fact that we are on the boundary.  
The value $5\pi/3$ is prohibited by the fact that they is no yellow--green simplicial  path of length $5\pi/3$ inside in the prescribed set of rings (by simplicial we mean both of whose end points are vertices of the ring).   
\end{proof}

Let us consider the finite $\diamond$-components first. 

\begin{lemma}\label{L - n equals 1 or 2}
Every finite $\diamond$-component of $P$ is a parallelogram of size $m\times n$ such that 
 $ \min(m,n)\leq 2$.
\end{lemma}

\begin{proof}  
The fact that finite $\diamond$-components are parallelograms follow from Lemma \ref{L - components are losenges}. If $m\times n$ denotes the size of such a component with $m\leq n$, then we claim that $m\leq 2$. This is a combinatorial problem which  can be solved in the 3-neighbourhood of the component. Note that the 3rd ring has to be used at every point in the boundary, except at the two obtuse angles, where it is the 2nd ring that has to be used. 

The solution of the problem, assuming $m\geq 3$ towards a contradiction, is summarized in the following drawing:
\[
\includegraphics[width=6cm]{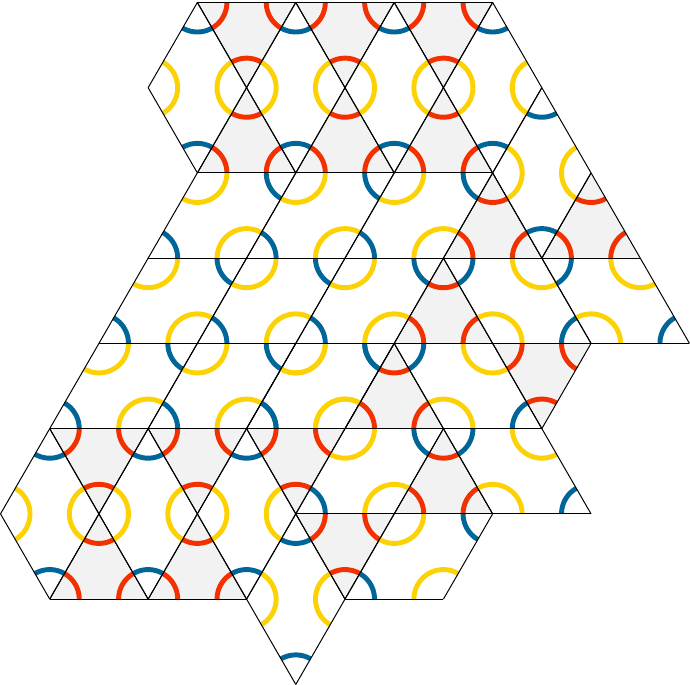}
\]
The rightmost part does not extend in $P$ given the prescribed sets of types. The drawing explicits one of several similar cases.
\end{proof}

In general, the prescribed sets in a ring puzzle problem will determine basic blocks (``tiles'') that fit together to build larger portions of the puzzle using the extension property. 

An interesting block for $\Aut(F_2)$ puzzles is the \emph{$w$-block}:
\[
\includegraphics[width=1.7cm]{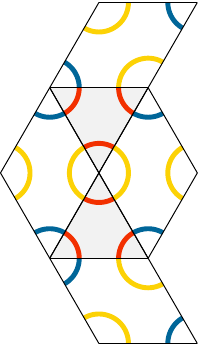}
\]
The $w$-block appears in the proof of Lemma \ref{L - n equals 1 or 2}. 

Observe that $w$-blocks are ``forward analytic'' in the following sense:

\begin{lemma}\label{L - w-blocks} Every $w$-block  in $P$ extends to a unique $w$-strip in $P$.  
\end{lemma}

By $w$-strip we mean the  strip (infinite at least on the right hand side): 
\[
\includegraphics[width=3.5cm,trim = 0cm 1.6cm 0cm 0cm]{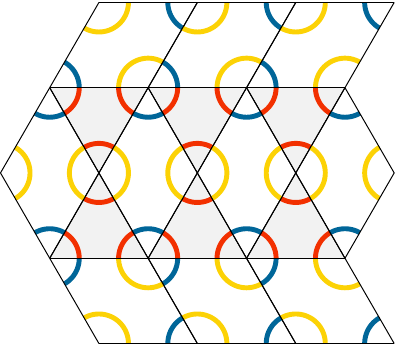}\ \ \cdots
\]
\vspace{1cm}

\noindent Lemma \ref{L - w-blocks} is a  consequence of Lemma \ref{L -NPC}. 

Note that two $w$-strips cannot intersect transversally in the forward direction.

\begin{lemma}\label{L - 2x2}
The $\diamond$-components of type $2\times 2$  in $P$ extend uniquely into a bi-infinite strip of height 4 of the form
\[
\includegraphics[width=8cm]{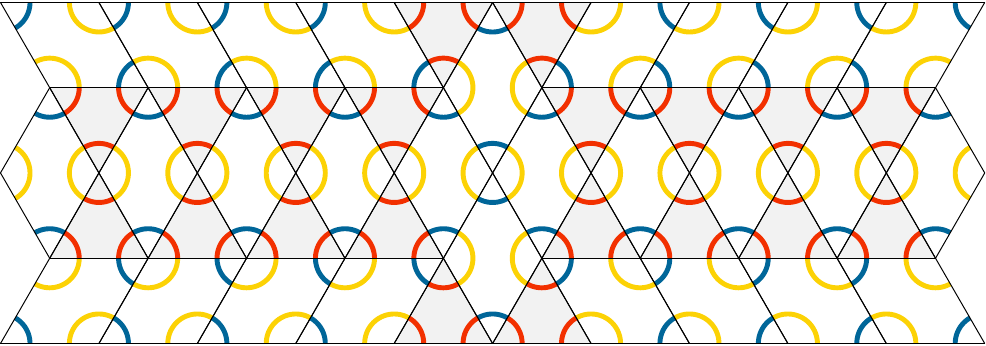}
\]
\end{lemma}

\begin{proof}
This is a  consequence of Lemma \ref{L - w-blocks}.
\end{proof}

We call this strip the double $w$-strip.

\begin{lemma}\label{L - 2xn puzzle}
Finite $\diamond$-component of type $2\times n$ with $n\geq 3$ are contained in a unique puzzle. 
\end{lemma}

\begin{proof} A $\diamond$-component $C$ of size $2\times n$ with $n\geq 3$
\[
\includegraphics[width=3cm]{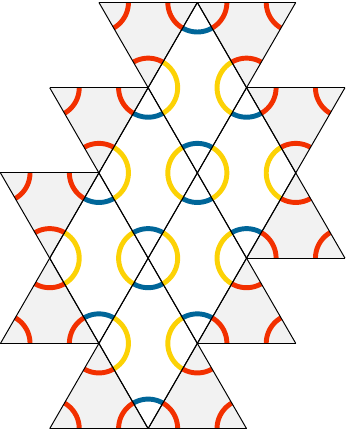}
\]
admits a unique extension of the form
\[
\includegraphics[width=9cm]{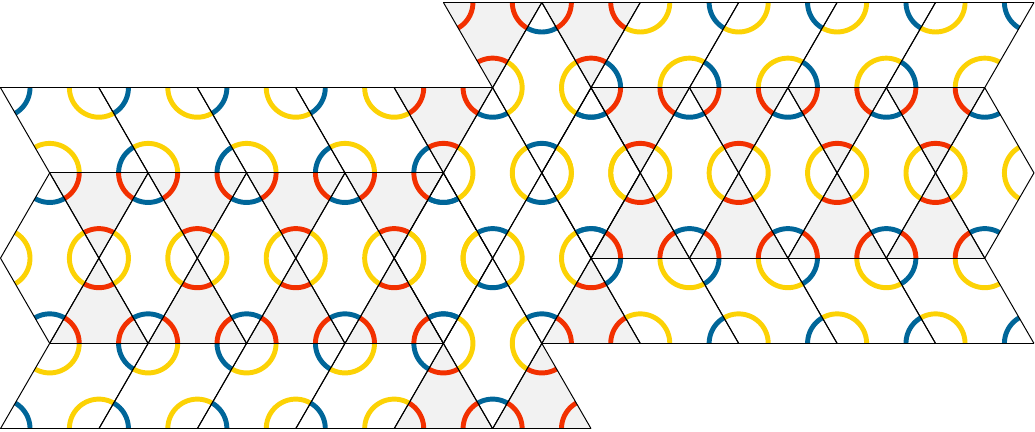}
\]
by Lemma \ref{L - w-blocks}.

Furthermore using the extension property we have a unique upward extension (and by symmetry  a unique downward extension) of the strip. This in turn exhibits a $w$-strip in the north-east direction (and by symmetry in the south-west direction) that edges an acute sector along with the horizontal $w$-strip,  implying that the horizontal strip belongs to at most one puzzle. A construction by induction then shows that the puzzle exists.
   
We will call this puzzle the \emph{puzzle of type $2\times n$}. \end{proof}

Here is a portion of the puzzle of type $2\times 4$:
\[
\includegraphics[width=9cm]{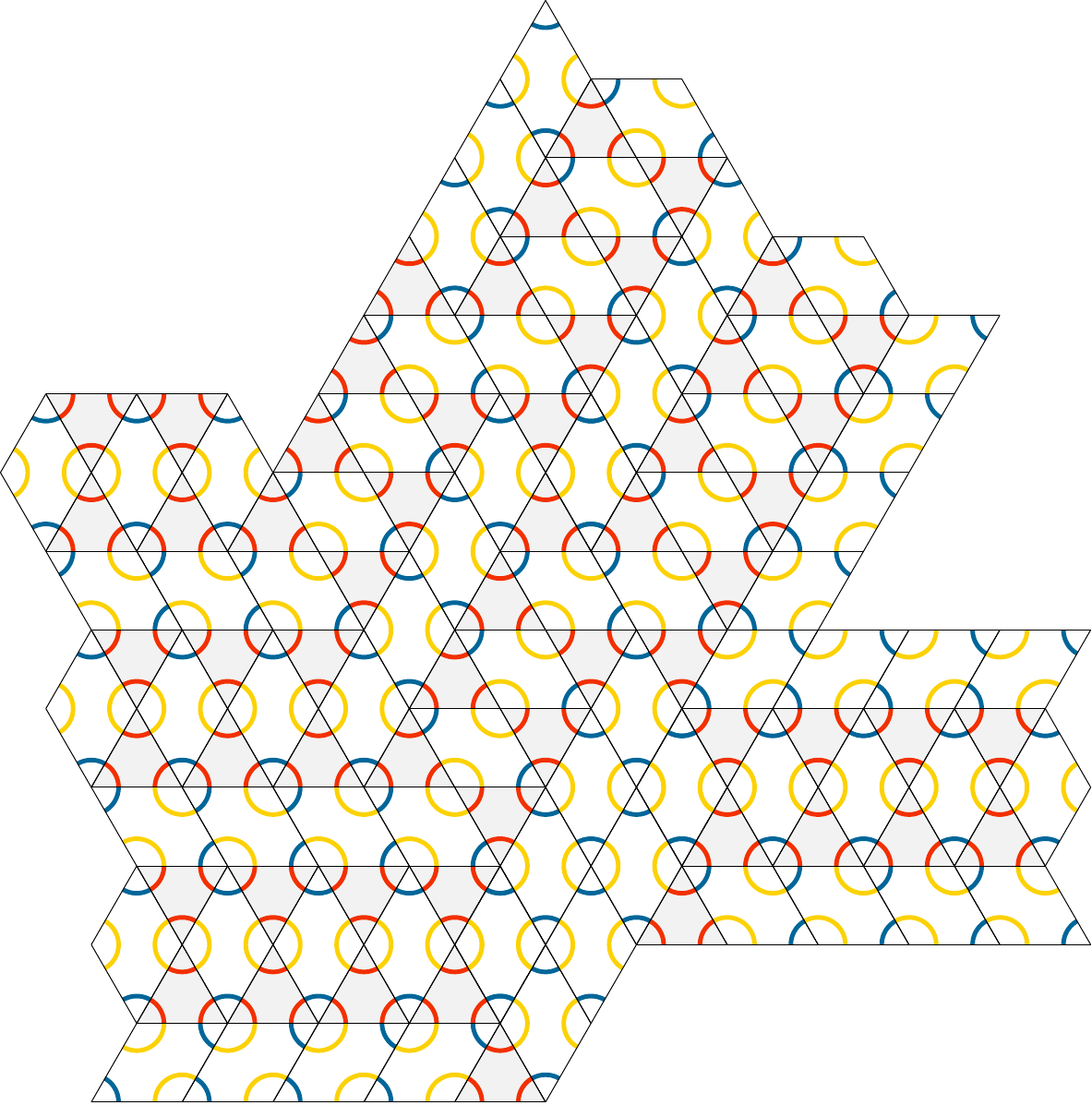}
\]

\begin{remark} Lemma \ref{L - 2x2} can be understood as the limiting case ``$n\to 2$'' of  Lemma \ref{L - 2xn puzzle}. 
\end{remark}

\begin{lemma}\label{L - component 1xn}
Assume that $P$ contains a finite $\diamond$-component $C$ of type $1\times n$ where $n\geq 3$. Then $C$ is included in a  strip $S$ of height $n$ in $P$ (infinite both to the left and to the right) of the form:
\[
\includegraphics[width=10cm]{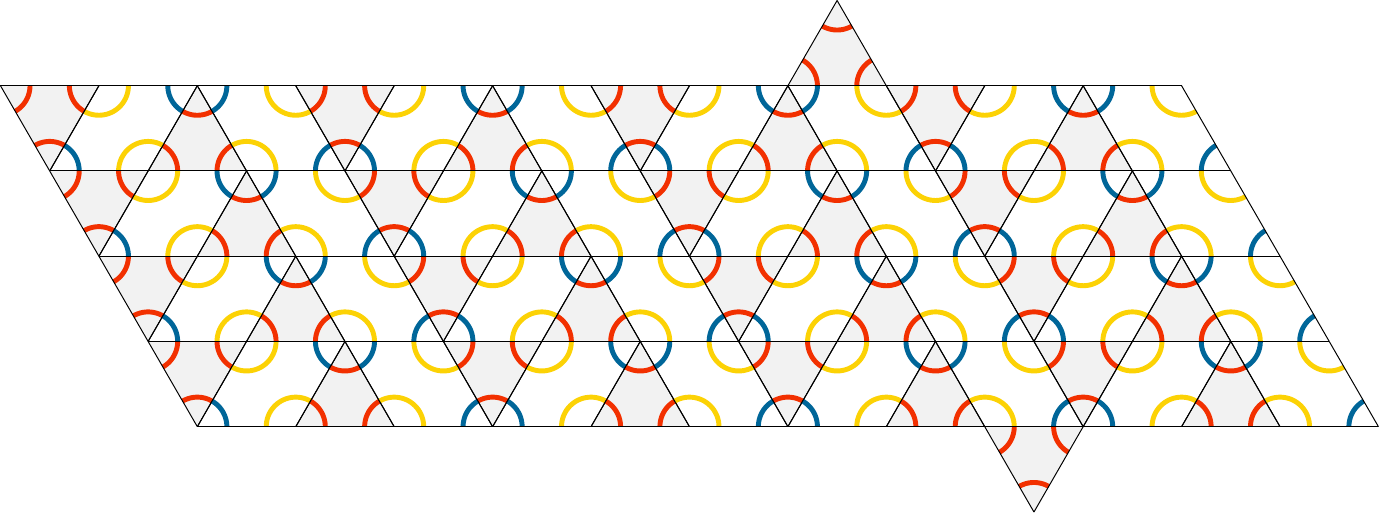}
\]
Furthermore, the following two conditions are satisfied: 
\begin{itemize}
\item [a)] the $\diamond$-components of $S$ parallels to $C$  are all of type $1\times n$, except for at most one component of type  $2\times n$. If there is such a component, then $P$ is the puzzle of type $2\times n$, otherwise $S$ is as indicated in the figure above. 
\item [b)] every $\diamond$-component of $S$ remains a $\diamond$-component in $P$ (otherwise said, it is adjacent to triangles in $P\setminus S$ on both sides, as indicated in the figure). 
\end{itemize}
\end{lemma}

The  strip described in the figure above is called the \emph{bi-infinite strip of type $1\times n$} (all $\diamond$-component of the strip are of type $1\times n$).

\begin{proof}
Starting with a $\diamond$-component in $P$ of size $1\times n$
\[
\includegraphics[width=5cm]{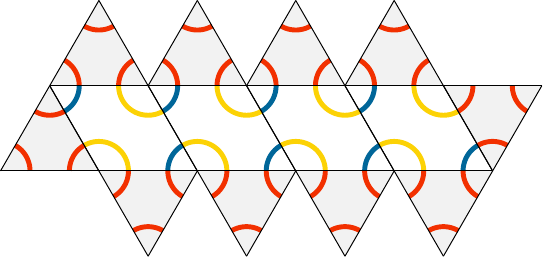}
\]
there is a unique extension of the form
\[
\includegraphics[width=5cm]{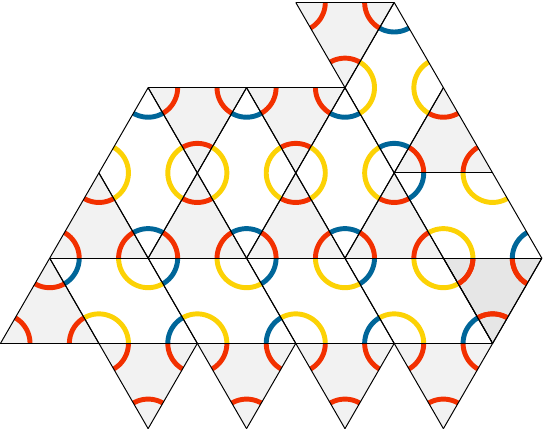}
\]
which by induction  form a half strip $S_\uparrow\supset C$ in the upward direction, together with the right 1-neighbourhood of it. The same argument applies  in downward direction by symmetry, which delivers the indicated strip $S = S_\uparrow\cup S_\downarrow \supset C$ together with a 1-neighbourhood of it on both sides. 

Inspection of the neighbourhood   shows that the $\diamond$-components parallel to $C$ are $\diamond$-component in $P$, which therefore are of the form $m\times n$ for $m=1,2$.

We have to prove that a component of type $2\times n$ can only appear once. 
This can be done directly noticing that the puzzle of type $2\times n$ actually contains only one $\diamond$-component of type $2\times n$, or with the following direct proof. Suppose we have two components, without loss of generality we may assume that they are separated only by components of type  $1\times n$.
\[
\includegraphics[width=10cm]{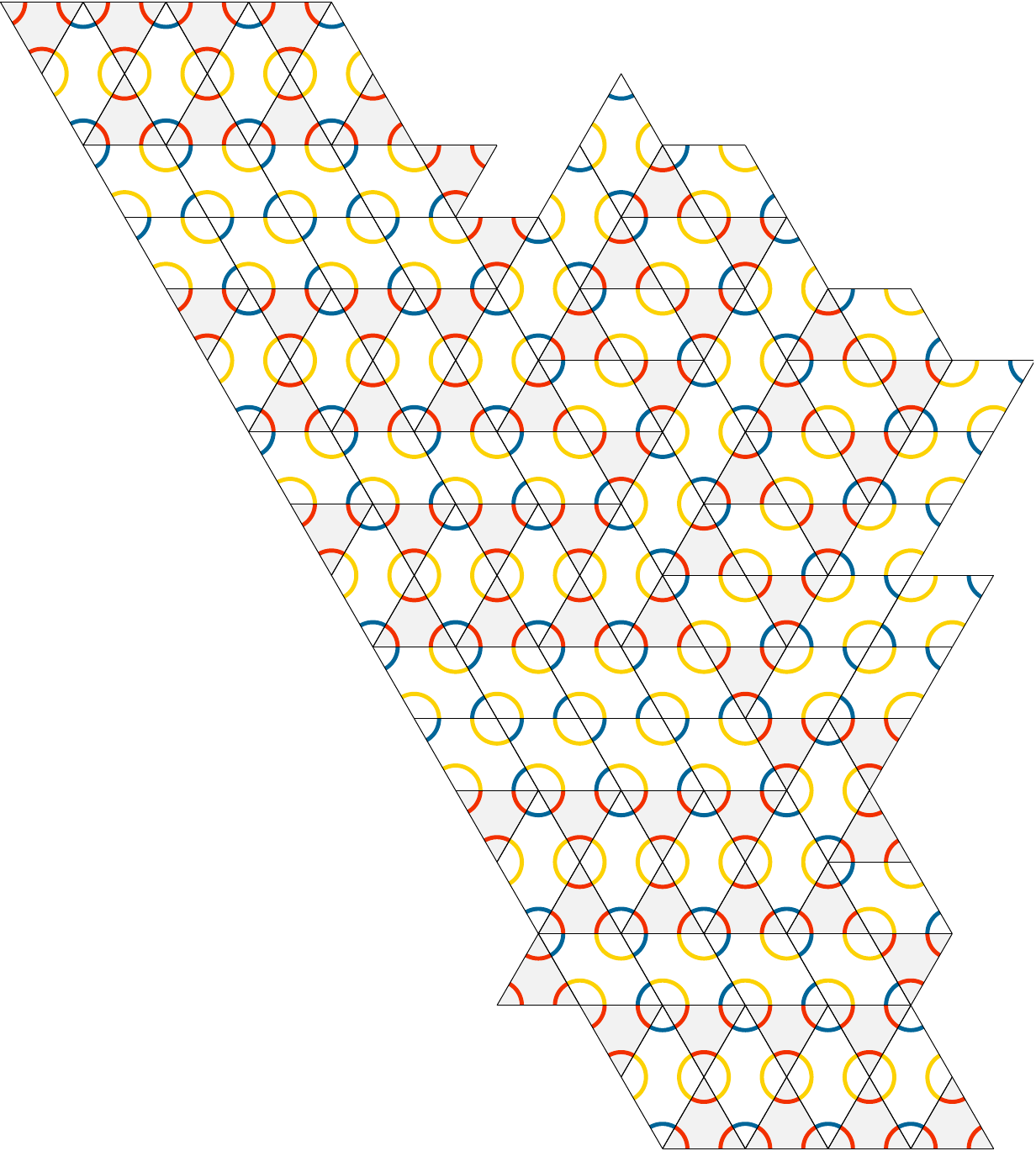}
\]
Note that the eastern neighbourhood of any $\diamond$-component of type $2\times n$ contains two $w$-blocks whose \emph{forward}  direction intersect transversally. 

By Lemma \ref{L - w-blocks}, $w$-blocks extend uniquely  in the forward direction. This contradicts  the fact  that $w$-strips do not intersect transversally, and proves the result.  

The figure above illustrates the argument. We have  indicated the completion in the case of two $\diamond$-components of type $2\times n$  separated by precisely one $\diamond$-component of type $1\times n$. In this case a direct local contradiction appears in the 1-neighbourhood of the upmost component of type $2\times n$.
 \end{proof}

\begin{lemma}\label{L- 2x1}
Assume that every  component is finite and of type $1\times n$ with $n=1,2$. 
Then every finite component is of type $1\times 2$, and corresponds to a  unique corresponding $\Aut(F_2)$ puzzle.
\end{lemma}

\begin{proof} It is easy to check that there is no puzzle in which every $\diamond$-component is of type $1\times 1$. The proof of the lemma, starting from a component of type $1\times 2$, is then in the same spirit as that of Lemma \ref{L - component 1xn}, given the fact that every component has type $1\times n$ with $n\leq 2$.  The resulting puzzle of type $1\times2$ is
\[
\includegraphics[width=10cm]{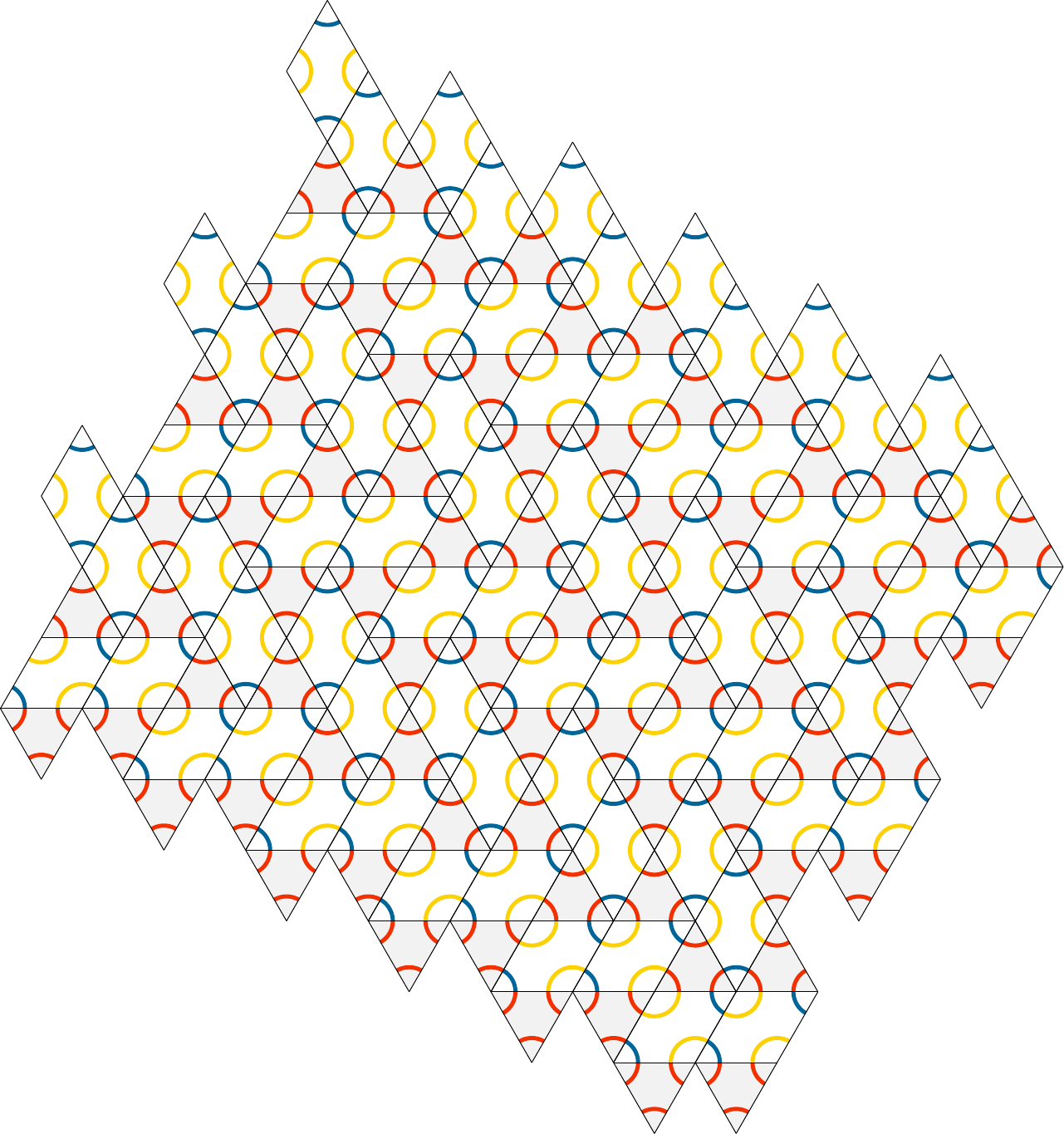}
\]
\end{proof}

We now turn to infinite $\diamond$-components. 

\begin{lemma}\label{L - infinite a b c cases}
An infinite $\diamond$-component is either
\begin{itemize} 
\item[a)] the  tessellation of the plane by lozenges (called the $\diamond$-plane) or half of it (called the half $\diamond$-plane)
\item[b)] a bi-infinite strip of lozenges ($\diamond$-strip) or half of it (semi-infinite $\diamond$-strip)
\item[c)] a $\diamond$-sector (with either an acute or an obtuse angle)
\end{itemize}
\end{lemma}

\begin{proof}
This follows from Lemma \ref{L - components are losenges}. 
\end{proof}

Every (semi-infinite) $\diamond$-strip of lozenges is bounded by a (semi-infinite) 2-strip. In particular both half $\diamond$-planes and bi-infinite $\diamond$-strip give rise to  an alternating 2-strip/$\diamond$-strip  $\Aut(F_2)$ puzzles.

Note that:

\begin{lemma}
If $P$ contains a semi-infinite $\diamond$-strip $S$ of height $n$, then $n=1$ or $n=2$.
\end{lemma}

The proof  is similar  to that of Lemma \ref{L - n equals 1 or 2}. 

\begin{lemma}\label{L - 2 x infty puzzle}
The semi-infinite $\diamond$-strips of type $2\times \infty$ can be extended in precisely two ways: a puzzle called the \emph{star puzzle of type $2\times \infty$}, or a  half-puzzle whose boundary is the following strip of height 3:
\[
\includegraphics[width=8cm]{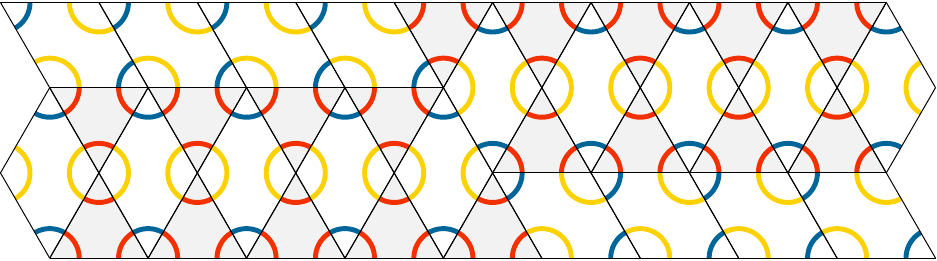}
\]  
We refer to the above strip as ``the 3-strip'', and 
the half-puzzle containing it will be called the \emph{half puzzle of type $2\times \infty$}.
\end{lemma}

\begin{proof}
Take a semi-infinite (horizontal) $\diamond$-strip $S$ of height $2$.
The proof of Lemma \ref{L - component 1xn} shows that there exits a  unique sector extending $S$ ``in the upper direction''. This sector contains a double $w$-strip as follows:
\[
\includegraphics[width=8cm]{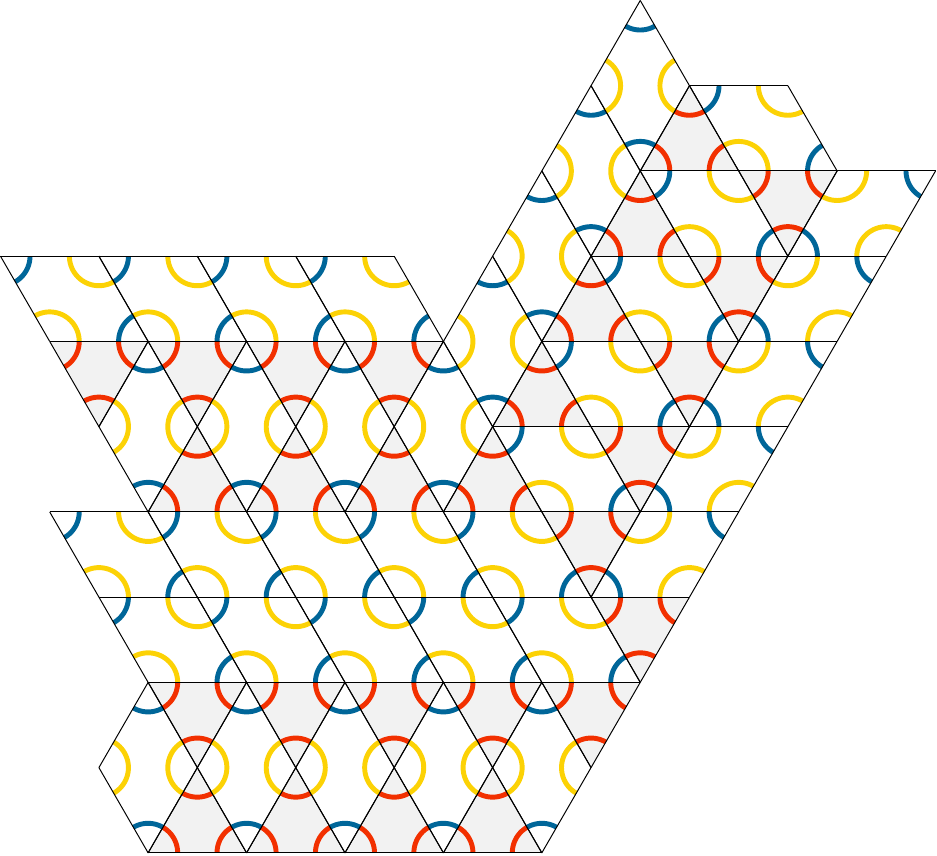}
\]
There are two  extensions of it, corresponding to the two ways a lozenge can be oriented with respect to adjacent the south-east triangles in the figure. 

We have already encountered the first extension. It can be derived from the proof of Lemma \ref{L - 2xn puzzle} or Lemma \ref{L - component 1xn} and appears as a ``half of'' the limit (in the Gromov-Hausdorff sense, say) of  the puzzle of type $2\times n$ as $n\to \infty$:
%\[
%\includegraphics[width=11cm]{onexn_parallelograms_1_2_1.pdf}
%\]
\[%smaller drawing
\includegraphics[width=8cm]{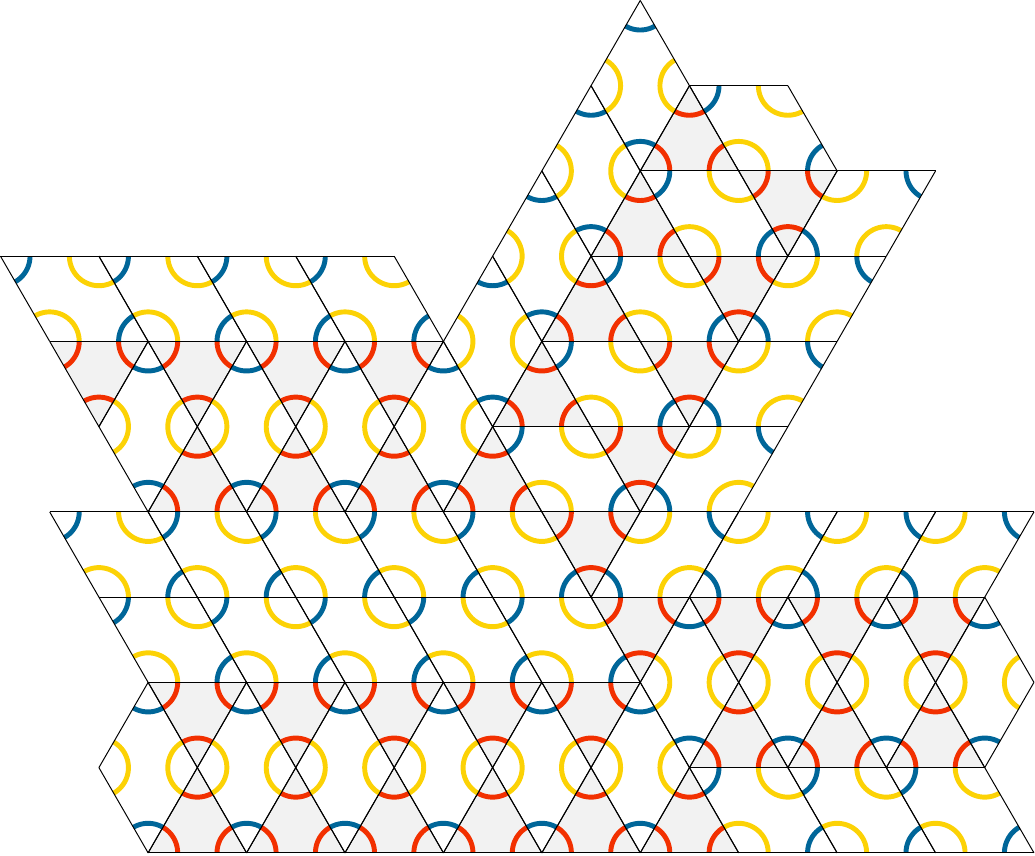}
\]
The lowermost 3-strip is the  3-strip. 

The second extension, which leads to a unique puzzle, is 
\[
\includegraphics[width=8cm]{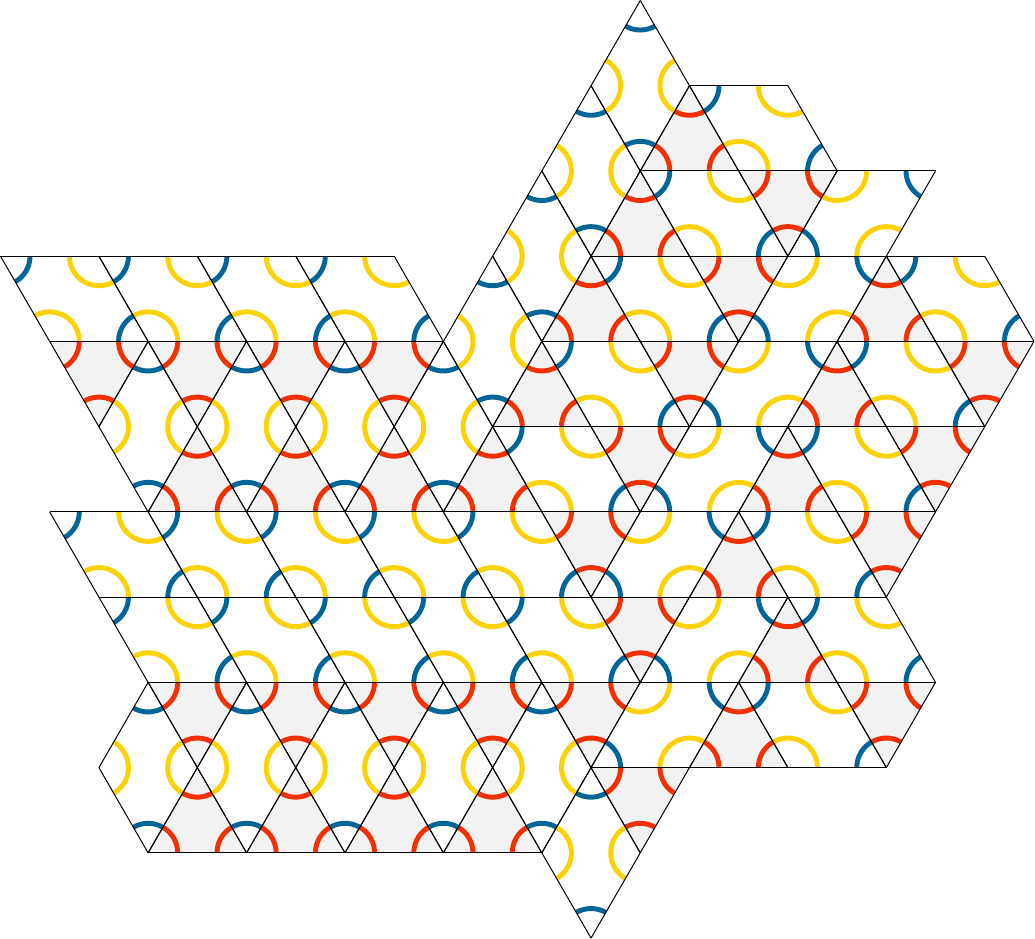}
\]
That it  can be extended into at most one puzzle follows by Lemma \ref{L - w-blocks} (and Lemma \ref{L -NPC}), and it is easy to show that this puzzle (the star puzzle of type $2\times \infty$) exists by induction.
\end{proof}

In particular, it appears that the semi-infinite $\diamond$-strip of type $2\times \infty$  embeds twice (at most) in an $\Aut(F_2)$ puzzle.

\begin{lemma}\label{ L - component 1xinfty}
If $P$ contains a semi-infinite $\diamond$-strip of type $1\times \infty$ then  one of the following holds:
\begin{itemize}
\item $P$ is the puzzle of type $2\times n$ where $n\geq 3$
\item $P$ contains the half-puzzle of type $2\times \infty$
\item $P$ is the star puzzle of type $2\times \infty$
\item $P$ contains the double $w$-strip (see Lemma \ref{L - 2x2})
\item $P$ contains the half-puzzle of type $1\times n$ (defined below) whose boundary strip of type $1\times n$ for some $n\geq 3$
\item $P$ is the \emph{$V$-puzzle}:
\[
\includegraphics[width=10cm]{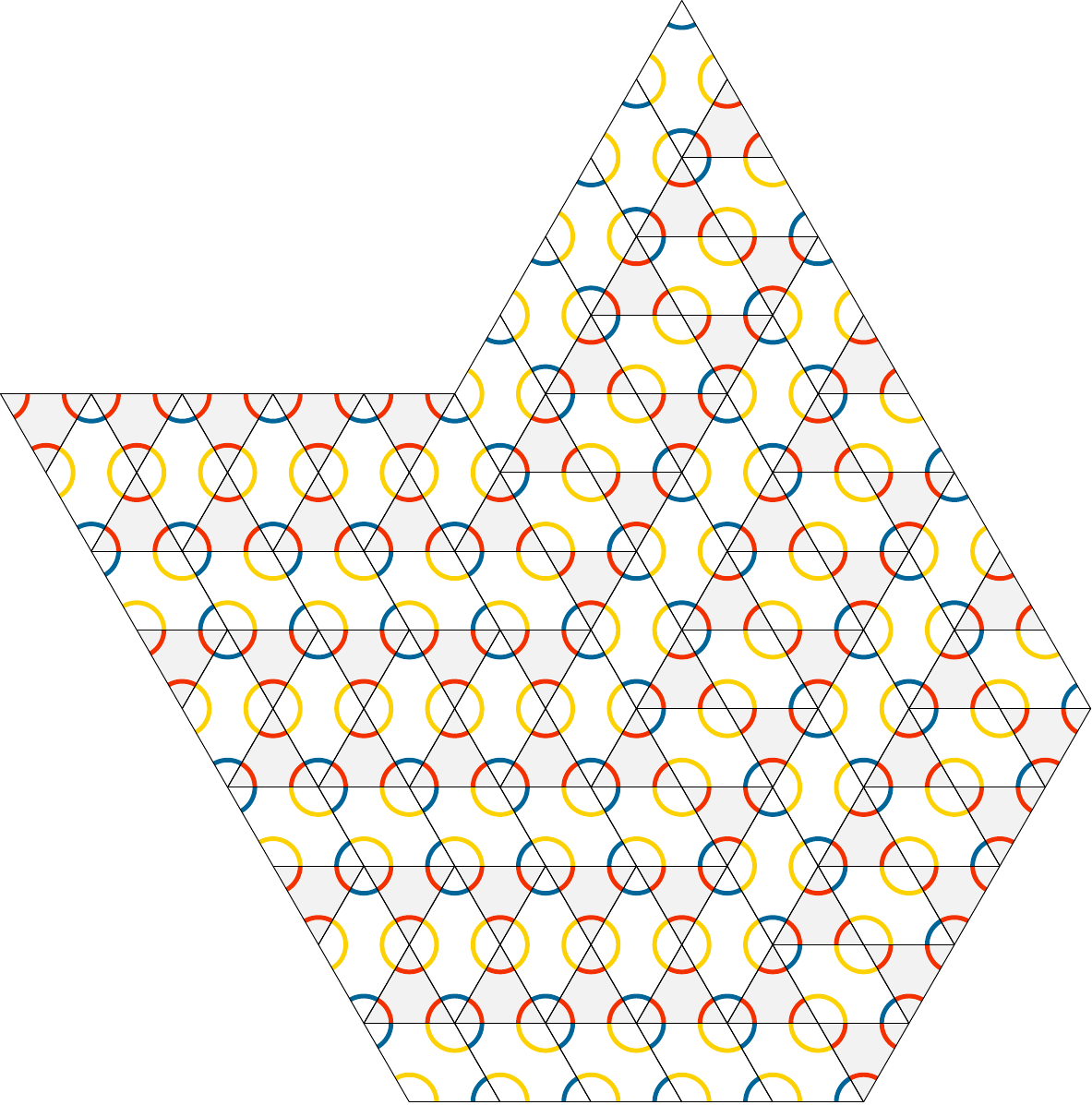}
\]
\end{itemize} 
\end{lemma}

\begin{proof} Take a semi-infinite (horizontal) $\diamond$-strip $S$ of height 1.
Again  as in the proof of Lemma \ref{L - component 1xn} one can construct a quarter puzzle $P_\uparrow$ extending $S$ ``in the upper direction''   which contains parallel (semi-infinite) 2-strips and (by assumption) semi-infinite $\diamond$-strip $S$ of height 1.

The lower part of $P_\uparrow$ admits a unique extension as follows:
\[
\includegraphics[width=5cm]{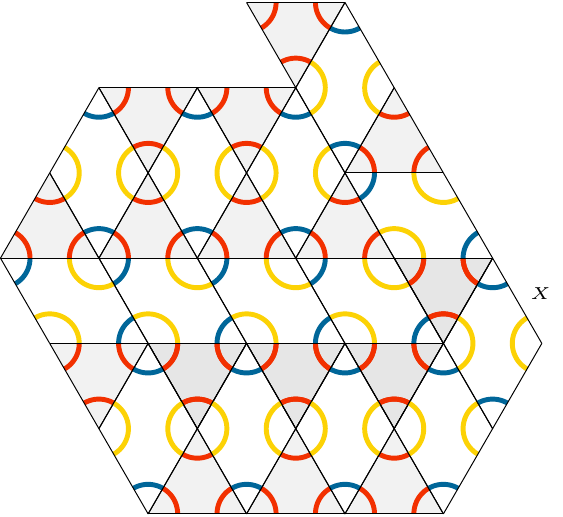}
\]
(In the picture the semi-infinite strips are horizontal.)

If the piece at $X$ is a triangle, then we have a $\diamond$-component of type $2\times n$ below it heading south-east. There are 4 cases: 
\begin{itemize}
\item If $n=1$ then $P_\uparrow$ extends uniquely east using a 3-strip. The resulting half-puzzle $P_{\uparrow,\to}$ piles up horizontal 3-strips.
\item If $n=2$ then by Lemma \ref{L - 2x2} the quarter puzzle $P_\uparrow$ extends uniquely east using the 4-strip (compare Lemma \ref{L - extension 4-strip}).
\item If $3\leq n<\infty$ then  $P$ is the puzzle of type $2\times n$ (see Lemma \ref{L - 2xn puzzle})
\item Otherwise,  $n=\infty$, and Lemma \ref{L - 2 x infty puzzle} shows that either $P$ is the star  of type $2\times \infty$ or contains the half-puzzle of type $2\times \infty$.  
\end{itemize}

We now assume that the piece at $X$ is a square and extend the portion of the puzzle uniquely as follows:
\[
\includegraphics[width=5.5cm]{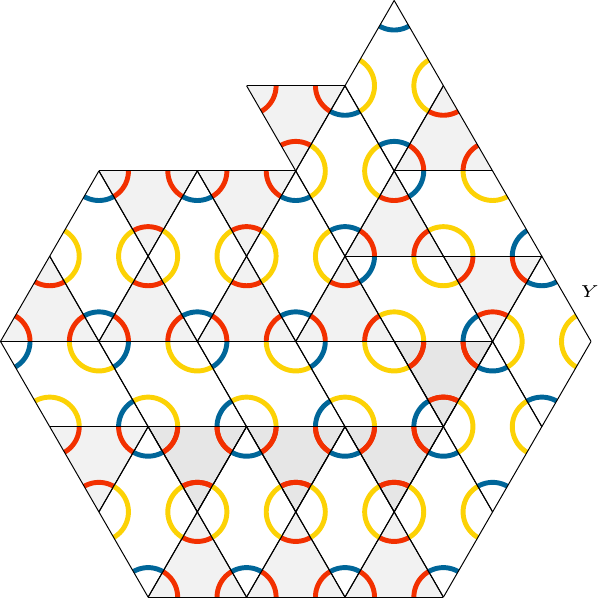}
\]
If the piece at $Y$ is a triangle, then by Lemma \ref{L - n equals 1 or 2} the puzzle contains a $\diamond$-component of type $m\times 3$ with $m=1,2$, 
and we can apply Lemma \ref{L - component 1xn} to find either that $P$ is the puzzle of type $2\times 3$, or if $n=1$ that $P$ contains the half-puzzle of type $1\times 3$.
\[
\includegraphics[width=10cm]{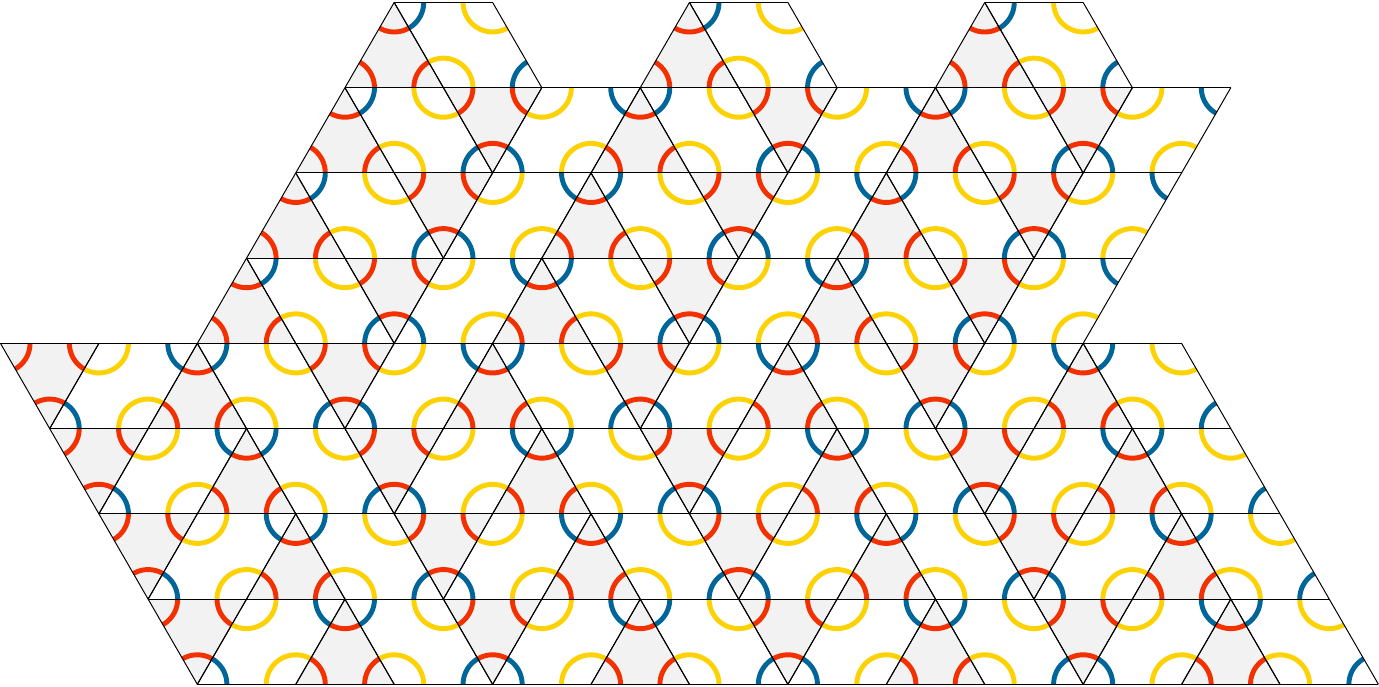}
\]
\begin{center}
The puzzle of type $1\times n$ (for $n=4$)
\end{center}

The conclusion is identical for every $n=3,4,\ldots$ If $n=\infty$ we obtain a sector in $P$ which is part of the $V$-puzzle (and corresponds to the limit $n\to\infty$). Since the horizontal strips alternate 2-strips and $\diamond$-strips of height 1, this sector in $P$ extends uniquely to the $V$-puzzle. 
 \end{proof}

Let us now turn to the $\diamond$-sectors, both of which are limiting cases of half $\diamond$-strips as the height increases. Recall however that $\diamond$-strips in flat can only have height 1 or 2.

\begin{lemma}\label{L - acute}
If the puzzle contains a sector with acute angle, then it contains a component of type $2\times n$ for some $n= 1,\ldots, \infty$ (if $n=\infty$ this is the semi-infinite $\diamond$-strip of type 2). 
\end{lemma}

\begin{proof}
Acute sectors extend uniquely (up to symmetry along the bisector) to the following finite puzzle: 
\[
\includegraphics[width=7cm]{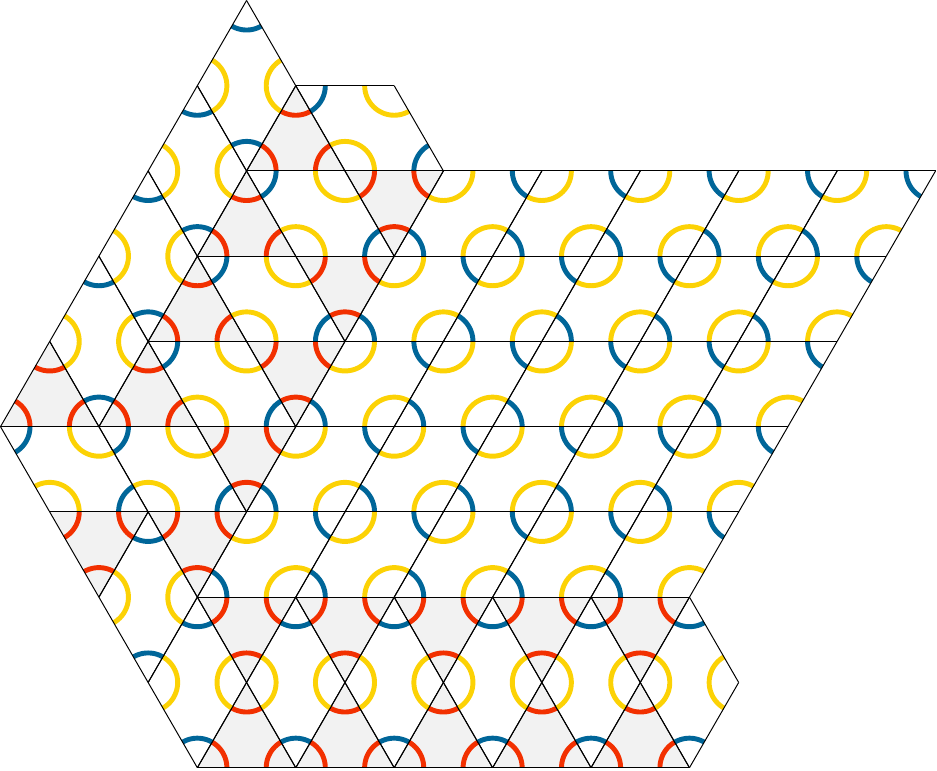}
\]
This puzzle  contains (at the bottom-left) a $\diamond$-component of type $2\times n$ with $n=1,\ldots,\infty$.  If $n=1$ then the sector extends uniquely to a half plane containing the 3-strip. If $n=2$ then it extends uniquely to a half-plane containing the 4-strip. Otherwise 

 By Lemma \ref{L- 2x1}, we have $n\geq 2$ and by Lemma \ref{L - 2x2} we have $n\geq 3$. Applying Lemma \ref{L - component 1xn} gives the conclusion. 
\end{proof}

\begin{lemma}\label{L - obtuse}
There exists a unique puzzle containing a sector with obtuse angle.
\end{lemma}

\begin{proof}
Obtuse sectors extend to: 
\[
\includegraphics[width=7cm]{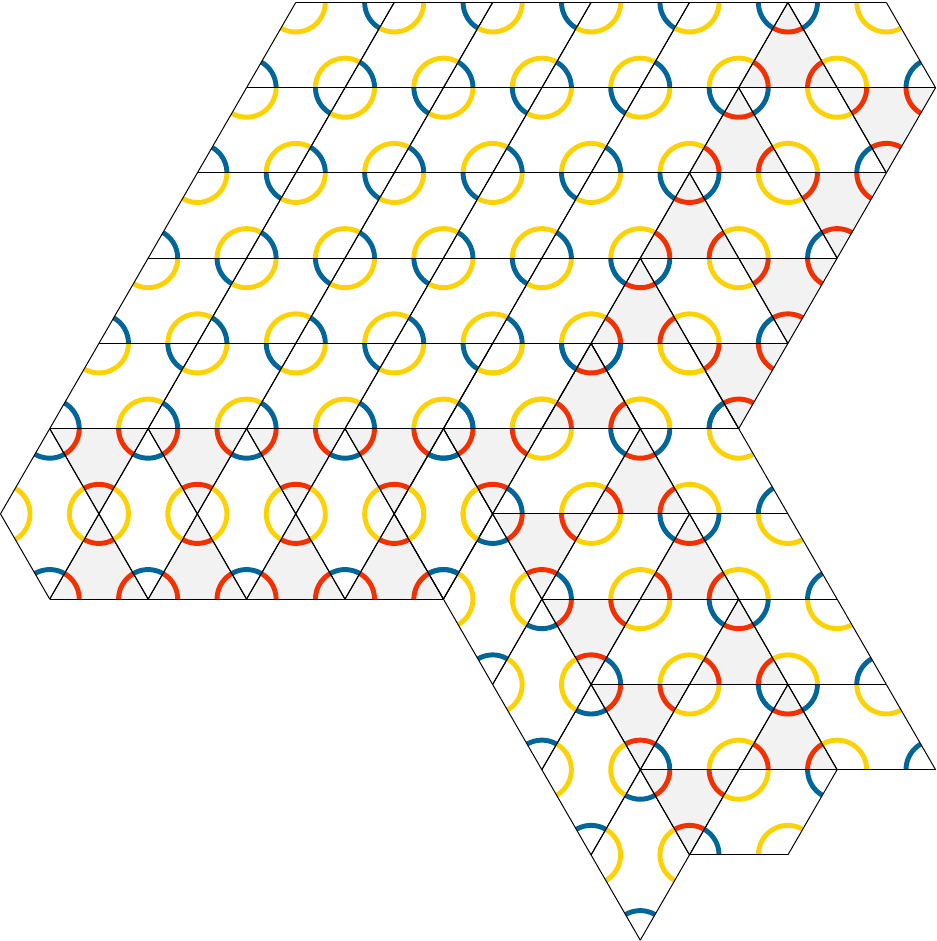}
\]
which creates a $w$-strip. Existence and uniqueness of the puzzle is proved by induction using Lemma \ref{L -NPC}   as earlier. 
%\[
%\includegraphics[width=9cm]{sector_obtuse_2.pdf}
%\]
\end{proof}

Let us then classify extensions of the 3-strips and the double $w$-strips. 
 
\begin{lemma}\label{L -  3-strip}
If $P$ contains the 3-strip $S$, then $P$ contains only parallels 3-strips and at most two copies of the half puzzle of type $2\times \infty$.
\end{lemma}

 The \emph{3-strip puzzle} is the puzzle made only of parallels 3-strips.  Every  other puzzle is called the 3-strip puzzle of height $h$ where $h\geq 1$ is the number of parallels 3-strips it contains. We call \emph{half 3-strip puzzle} the puzzle containing exactly one half-puzzle of type $2\times \infty$. Every 3-strip puzzle of height $h<\infty$ contains two such puzzles. 

\begin{lemma}\label{L - extension 4-strip}
There are precisely two puzzles containing the double $w$-strip. 
\end{lemma}

\begin{proof}
Note that, using the symmetry along the vertical axis, the double $w$-strip extends uniquely as follows:
\[
\includegraphics[width=8cm]{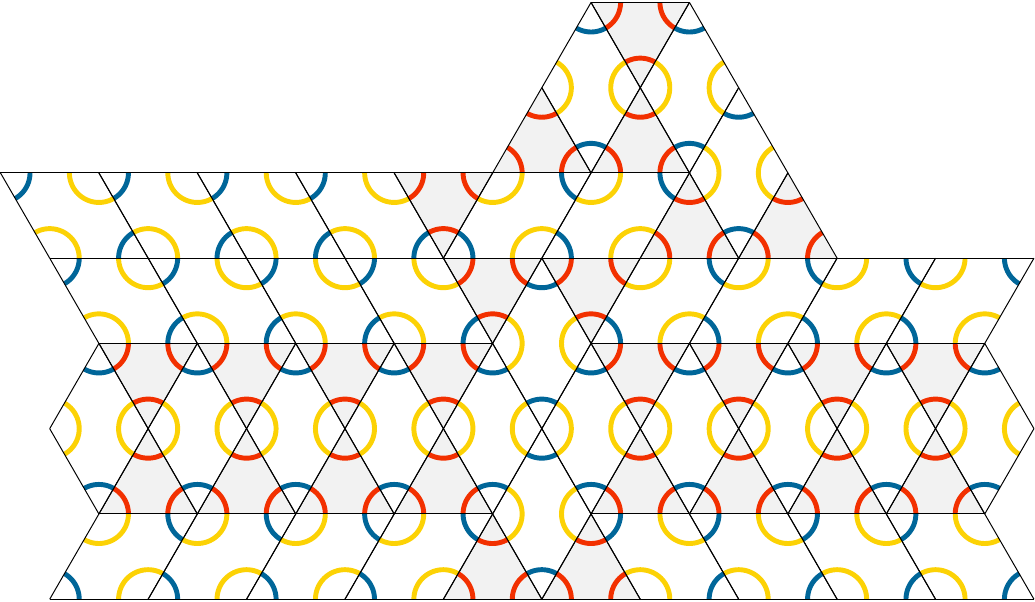}
\]
By Lemma \ref{L - w-blocks} this can be extended north in a unique way. 
\end{proof}
 
 The two puzzles in Lemma \ref{L - extension 4-strip} will be called the opposite acute sector puzzle and the adjacent acute sector puzzle.  The corresponding half puzzle is: 
\[
\includegraphics[width=11cm]{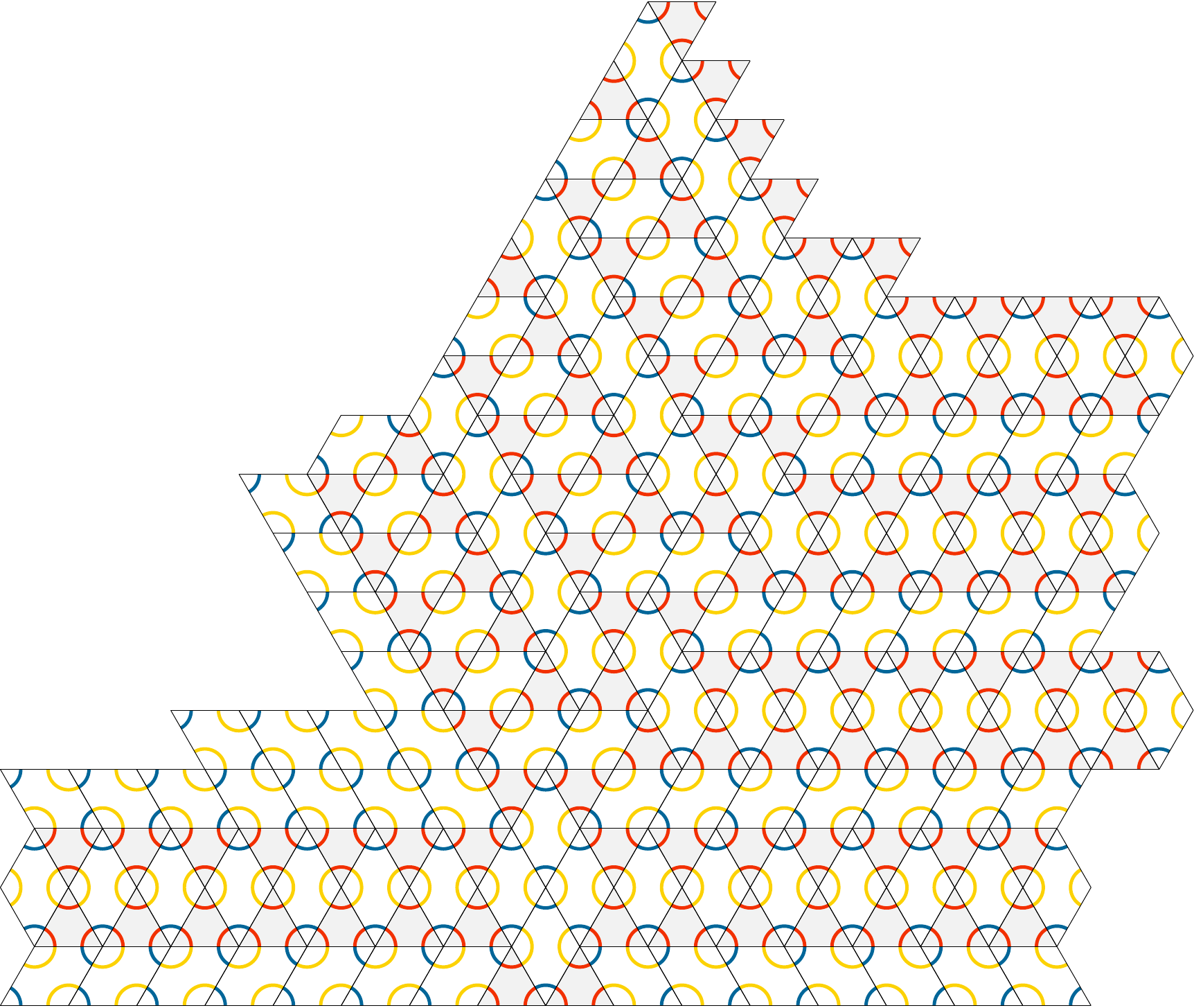}
\]

 \begin{theorem}\label{T - Puzzle classification} The $\Aut(F_2)$ puzzles come in 4 infinite series and  9 exceptional puzzles not belonging to the series. 
 
 The 9 exceptional puzzles are:
 \begin{enumerate}[a)]
 \item the $\diamond$-puzzle
 \item the $2\times 1$-puzzle (see Lemma \ref{L- 2x1})
\item the star puzzle of type $2\times \infty$  
 \item the 3-strip puzzle and the half 3-strip puzzle
  \item the opposite acute sector puzzle and the adjacent acute sector puzzle  (see Lemma \ref{L - 2x2} and Lemma \ref{L - extension 4-strip}).
 \item the obtuse sector puzzle (see Lemma \ref{L - obtuse})
 \item the $V$-puzzle
 \end{enumerate} 

The 4 infinite series are:
 \begin{enumerate}[A)]
 \item the alternating 2-strip/$\diamond$-strip series. This contains three subseries:
 \begin{itemize}
  \item the \emph{bounded 2-strip/$\diamond$-strip} series. They are parametrized by  biinfinite sequences $(\cdots,n_{-1},n_0,n_1,\cdots )$ recording the heights of the $\diamond$-strips in the given order, where $n_i\geq 1$ is arbitrary.  There are uncountably many such puzzles. 
  \item The \emph{sided 2-strip/$\diamond$-strip} series. The \emph{1-sided} series is parametrized by right infinite sequences $(n_0,n_1,\cdots )$ of heights of $\diamond$-strips, where $n_i\geq 1$ is arbitrary, giving uncountably many puzzles. The \emph{2-sided} series is parametrized by finite sequences
$(n_0,n_1,n_r)$ where $r$ is finite but arbitrary. This gives countably many puzzles. 
\end{itemize}

  \item the $1\times n$ series. They are constructing using the strip of type $1\times n$ (Lemma \ref{L - component 1xn}) and can be encoded by the biinfinite sequence  $(\cdots,n_{-1},n_0,n_1,\cdots )$
of heights here $n_i\geq 1$ is arbitrary, and at least one $n_i\geq 3$.  There are uncountably many such puzzles. This series also splits into 3 subseries,  the bounded $1\times n$ series and the sided $1\times \infty$ series, which can either be one-sided  or two-sided. 

 \item the $2\times n$ series, where $n=3,4, \ldots (n <\infty)$.  There is a unique puzzle for every $n\times 2$ (Lemma \ref{L - 2xn puzzle}).
 
 \item the 3-strip series of height $h\geq 1$ (Lemma \ref{L -  3-strip}).
 \end{enumerate}
\end{theorem}
 
Several of the exceptional puzzles appear as limits of elements of the infinite series. For example, a), \ c), and \ d) can be seen as limiting cases of A), \ C), and D) respectively, and b) can be seen as a limit of $B$ when ``$n\to 2$''. The $V$-puzzle g) can also be seen as a limiting case for B), and D) can be seen as a limiting case for C) when $n\to \infty$. The obtuse sector puzzle, on the other hand, does not appear to be a limit --- and similarly for e). 

\begin{proof}
Let $P$ be an $\Aut(F_2)$ puzzle. We assume that $P$ is not one of the exceptional puzzles and prove that it belongs to one of the series. 

Note that $P$ does not contain an obtuse sector. If it contains an acute sector then it contains a strip of type $2\times n$. If it contains a semi-infinite strip of type $2\times \infty$ and since it is not the star puzzle of type $2\times \infty$ it contains the 3-strip. Not being in d) it belongs to D).

If $P$ contains a semi-infinite strip of type $1\times \infty$ but no semi-infinite strip of type $2\times \infty$, then it belongs to B) by Lemma \ref{ L - component 1xinfty}.

If $P$ contains a bi-infinite $\diamond$-strip then it belongs to $A$.

By Lemma \ref{L - infinite a b c cases} we may now assume that every $\diamond$-component of $P$ is finite. If has a component of type $2\times n$ then we are in case C). Otherwise every component is finite of type $1\times n$. Then this is case B).
\end{proof}

\begin{remark}\label{R - marked puzzles}
One can organize puzzles into a ``space of puzzles'' as follows. A marked puzzle is a parametrization  $f\colon \IR^2\to P$ where $f$ is isometric. Define a valuation $v$ on the set $\MP'$ of all marked puzzles (with prescribed shapes and rings) by
\[
v(f,f')=\sup\{r,\ f(B(0,r))\stackrel{\text{simplicial isometric}}\simeq f'(B(0,r))\}
\]
with corresponding metric
\[
d(f,f')=e^{-v(f,f')}.
\]
The space of marked puzzles is 
\[
\MP:=\MP'/\{d=0\}.
\]
Some of the statements regarding convergence of puzzles can be interpreted in the space $\MP$. For example the obtuse sector puzzle ``is not a limit'' because it corresponds to an isolated point in $\MP$. Note that the space of puzzles is  compact if the sets of shapes and rings are finite. It is endowed with an action of $\IR^2$ given by 
\[
t\cdot f:= f\circ t^{-1}
\]
for $t$ a translation of $\IR^2$ turning $\MP$ into a lamination whose leaves correspond to puzzles with the given set of shapes and rings. 
\end{remark}

\section{Ring complexes and flat plane problems}\label{S - Nonplanar}
  
Ring puzzles, or portions thereof, can be  ``glued together'' to form ``nonplanar ring puzzles'', which are more complicated objects to understand.

Let us start with a 2-complex $X$ and say that $X$ is a \emph{ring complex}  if its faces are flat polygons in the Euclidean plane with coloured angles.  
  
  The \emph{rings} at a vertex of $X$ is set of all coloured circles of length $2\pi$ of small radius around that vertex. 

The \emph{type} of a ring complex is the prescribed set of coloured faces and rings (considered up to coloured isometry). 

We study  ring complexes of type $\Aut(F_2)$ as considered earlier:
  \[
T:=\left \{\includegraphics[width=2.7cm,trim = 0cm .35cm 0cm 0cm]{colored_shape_set.pdf}\ \     
 \ \ \ \ \   \includegraphics[width=.9cm,trim = 0cm .9cm 0cm 0cm]{ring_1.pdf}\ \  \includegraphics[width=.9cm,trim = 0cm .9cm 0cm 0cm]{ring_2.pdf}\ \ \includegraphics[width=.9cm,trim = 0cm .9cm 0cm 0cm]{ring_3.pdf}\right\}
\]
\vspace{.1cm}

\begin{lemma}
The group $\Aut(F_2)$ acts properly on a ring complex of type $\Aut(F_2)$ with compact quotient. 
\end{lemma}

\begin{proof}
The Brady complex  has the required properties. 
\end{proof}

Note that ring complexes are naturally endowed with a global metric, namely the length metric. They are not necessarily nonpositively curved. 

A flat plane in $X$ is an isometric embedding $\IR^2\inj X$. Every flat plane in a ring complex is a ring puzzle of the same type. 

The \emph{$\ZI^2$ embedding problem} for $X$ is the question of whether
\[
\IR^2\inj X\impl \ZI^2\inj \G
\]
whenever $\G\acts X$ properly with $X/\G$ compact. 
Gromov  asks for example  if this is true for  $X$ simplicial and nonpositively curved,  say, of dimension 2 (see \cite{gromov1993}).   When $X$ is nonpositively curved we have
\[
\ZI^2\inj \G\impl \IR^2\inj X
\]
by the  flat torus theorem and the ``flat closing conjecture'' asserts the embedding problem has a positive answer. In \textsection \ref{S - random autf2} and \textsection\ref{S - explicit examples 2} the groups are nonpositively curved.

\begin{definition}
Consider a simplicial strip  $S\simeq \IR\times [a,b]$ and a simplicial subset $L\subset S$. The strip $S$ is said to be  uniquely $L$-embeddable in $X$  if for any two simplicial embeddings $f,f'\colon S\inj X$ and any translation $g\colon S\to S$ 
\[
f'=f\circ g \text{ on } L \ \ \impl\ \  f'=f\circ g \text{ on } S. 
\]
(Note that $g$ is then necessarily a simplicial map.)
\end{definition}

Consider the following two 1-strips in a ring complex of type $\Aut(F_2)$:
\begin{itemize}
\item the $\diamond$-strip:
 \[
\includegraphics[width=12cm]{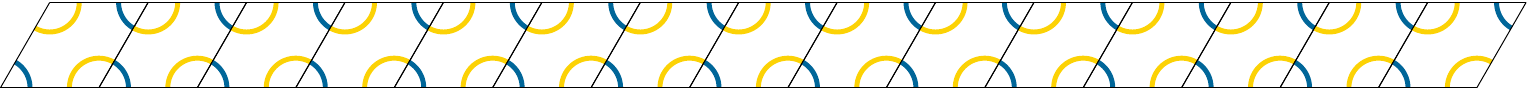}
\]
\item the $\triangle$-strip:
 \[
\includegraphics[width=12cm]{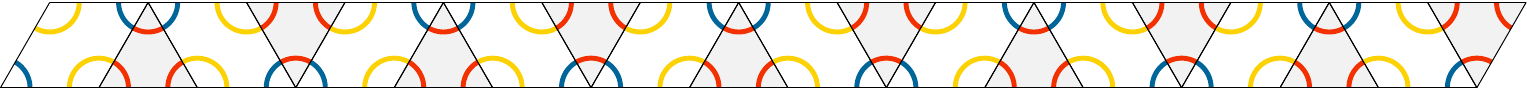}
\]
\end{itemize}

These strips are said to be \emph{uniquely $\diamond$-embeddable} if they are uniquely $L$-embeddable with respect to a lozenge. 

We say that $X$ is acylindrical if there is no embedding 
 \[\SI^1\times \IR\not \inj X\] 
 of a bi-infinite cylinder into $X$.

\begin{theorem}\label{Th - gromov Z2 puzzles} If in an acylindrical ring complex of type $\Aut(F_2)$  every $\diamond$-strip and every $\triangle$-strip is  uniquely $\diamond$-embeddable, then the $\ZI^2$ embedding problem has a positive solution.
\end{theorem}

The proof of Theorem \ref{Th - gromov Z2 puzzles} relies on two lemmas. The first lemma uses the classification. 

%Suppose that $X$ is a ring complex $X$  of type $\Aut(F_2)$ and fix  $\G\acts X$ with $X/\G$ compact.  

%\begin{lemma} 
%If $X$ contains a flat, then it contains a flat that contains a $\diamond$-strip or a $\triangle$-strip.  
%\end{lemma}

%\begin{proof}
%Let $\Pi$ be a flat. If $\Pi$ does not contain a $\diamond$-strip or a $\triangle$-strip, then it belongs to case c), d) e), f), D) in which case it contains a semi-infinite $\diamond$-strip, or it belongs to case C), in which case it contains a semi-infinite  $\triangle$-strip. Let $S$ be such a semi-infinite strip and $K$ be a fundamental domain for $\G\acts X$. Choose $x_n\to \infty$ vertices in $S$ and $s_n\in \G$ such that $s_nx_n\in K$.  For every $r>0$ we have that the flats $s_n\Pi$ coincide on $r$-neighbourhood of  $K$ for infinitely many $n$.  A diagonal argument delivers a flat \[
%\Pi_\infty:=\lim_{n \text{ in  a subsequence}} s_n\Pi
%\]
% with $x\in \Pi_\infty$ such that for every $r$ and every $n$ large enough (depending on $r$) the ball of radius $r$ and center $x$ in $\Pi$ coincide with the ball of radius $r$ and center $s_n^{-1}x$ in $\Pi_n$.    This shows that $\Pi_\infty$ contains either a (bi-infinite) $\diamond$-strip or a $\triangle$-strip.  
%\end{proof}

\begin{lemma} Suppose that $X$ is a ring complex $X$  of type $\Aut(F_2)$ and fix  $\G\acts X$ with $X/\G$ compact.  
If $X$ contains a flat, then it contains a flat that contains infinitely many parallel $\diamond$-strips, or a flat that contains infinitely many parallel $\triangle$-strips.  
\end{lemma}

\begin{proof}
We claim that if $X$ contains a flat, then it contains a flat which is isomorphic to either 
a $\diamond$-puzzle, 
a puzzle of uniformly bounded type $1\times n$ (including $n=2$), or
a puzzle of uniformly bounded alternating 2-strip/$\diamond$-strip type.
The lemma follows easily from this. 

The proof of the claim is a standard sort of compactness argument which relies on the classification. 
Let $\Pi$ be a flat and $K$ be a relatively compact fundamental domain for $\G\acts X$. 

Suppose for example that $\Pi$ contains a $\diamond$-sector $S$. Choose vertices $x_n$ in $S$ moving apart from the boundary  of $S$,   and $s_n\in \G$ such that $s_nx_n\in K$.  For every $r>0$ we can find infinitely many $n$ for which the flats $s_n\Pi$ coincide on $r$-neighbourhood of  $K$.  A diagonal argument delivers a flat \[
\Pi_\infty:=\lim_{n \text{ in  a subsequence}} s_n\Pi
\]
 with $x\in \Pi_\infty$ such that for every $r$ and every $n$ large enough (depending on $r$) the ball of radius $r$ and center $x$ in $\Pi$ coincide with the ball of radius $r$ and center $s_n^{-1}x$ in $\Pi_n$.  This shows that $\Pi_\infty$ is the $\diamond$-flat.
 
This covers cases a), d), e), f) and C) and D) of the classification. The same argument works if $\Pi$ contains $\diamond$-strips of arbitrary large height, showing that case B) of the classification leads to either a $\diamond$-flat, or $\Pi$ a uniformly bounded flat of type $1\times n$ (including $n=2$). This covers cases b) and B). 
 
In cases c) and g) it is easy to choose $x_n$ so that $\Pi_\infty$ is the alternating 2-strip/$\diamond$-strip flat where the $\diamond$-strip has height 1. Finally in case A), either the height is uniformly bounded, or we can choose $x_n$ so that $\Pi_\infty$ is the $\diamond$-flat.   
  \end{proof}

\begin{lemma} Suppose that $X$ is a ring complex $X$  of type $\Aut(F_2)$ and fix a proper action $\G\acts X$ with $X/\G$  compact.  
 Assume furthermore that 
\begin{itemize}
\item  $X$ is  acylindrical and uniformly locally finite 
\item $X$ contains a flat that contains infinitely many parallel copies of a simplicial strip $S$ 
\item $S$ is uniquely $L$-embeddable in $X$ for some non empty compact simplicial set $L$
\item  there are infinitely many distinct simplicial translates of $L$ in $S$.
\end{itemize} 
 Then $\ZI^2\inj \G$. 
\end{lemma}

\begin{proof} Let  $S$ denote the strip given with simplicial set $L\subset S$.
 
Let $\Pi$ be a flat in $X$ that contains infinitely many parallel copies of  $S$, say 
\[
f_k\colon S\inj X\  \text{ with } \ran f_k\parallel  \ran f_l,\ \   f_k\neq   f_l \ \ \ (k\neq l\geq 1)
\]
Note that $f_k\neq f_l$ on $L$ if $k\neq l$.

Since $X/\G$ is compact there exist elements $s_k\in \G$ such that 
\[
s_k\circ f_k = s_1 \circ f_1\text{ on } L
\]
for infinitely many $k\geq 1$. 

Since $S$ is uniquely $L$-extendable we have that
\[
s_k\circ f_k = s_1 \circ f_1\text{ on } S
\]
for infinitely many $k\geq 1$.

Choose $k$ large enough so that $\ran f_k\cap \ran f_1=\emptyset$ and write $f=f_1$, $f'=f_k$ and $s=s_k^{-1}\circ s_1$ so that $f' =s\circ  f$.

By assumption we can find infinitely many distinct translates of $L$ in $S$. Let $g_n\colon S\to S$ denote the corresponding translations $(n\geq 1)$. Since $X/\G$ is compact, there exist an index $n_0\geq 1$ and  elements $t_n\in \G$ such that 
\[
t_n\circ f\circ g_n =  f\circ g_{n_0} \text{ on } L
\]
for infinitely many $n\geq 1$.

Since $S$ is uniquely $L$-extendable we have that
\[
t_n\circ f\circ g_n =  f\circ g_{n_0} \text{ on } S
\]
for infinitely many $n\geq 1$.
 
Choose $n$ large enough so that $f\circ g_n(L)\cap  f\circ g_1(L)=\emptyset$ and write  $g=g_{n_0}\circ g_{n}^{-1}$ and $t=t_n$ so that we have the intertwining relation
\[
t\circ f =  f\circ g
\]
on $S$. Note that $g$ is a simplicial translation. Furthermore $t^m\circ f =  f\circ g^m$ for every $m\in \IZ$.
 
Denote by $f\vee f'\colon [a,b]\times \IR\to \Pi\subset X$ the simplicial strip lying ``between $f$ and $f'$'' in $\Pi$ (whose range is the convex closure of both strips in $\Pi$).  Thus there is $a<a'<b'<b$ such that $S$ is simplicially isomorphic  to both $[a,a']\times \IR$ and $[b',b]\times \IR$, and the restriction of $f\vee f'$ to these strips is $f$ and $f'$, respectively.

 Since $X$ is locally finite, there are only finitely many strips in $X$ isometric to $\ran f\vee f'$  containing $\ran f$. Therefore there exists $m$ and a translation $h\colon [a,b]\times \IR\to [a,b]\times \IR$ such that 
 \[
 t^m\circ (f\vee f')=(f\vee f')\circ h.
 \] 
Note that $h$ coincide with $g^m$ on $[a,a']\times \IR$, so we will replace $t$ with $t^m$ to obtain the relations
 \[
t\circ f =  f\circ h\text{ and } t\circ (f\vee f')=(f\vee f')\circ h.
\]
on $S\simeq [a,a']\times \IR$ and $[a,b]\times \IR$. This implies that 
 \[
t^q\circ f =  f\circ h^q
 \text{ and } t^q\circ (f\vee f')=(f\vee f')\circ h^q.
\]
on $S\simeq [a,a']\times \IR$ and $[a,b]\times \IR$, respectively, for every $q\in \IZ$, 
and in particular that 
\[
t^q\circ f' =  f'\circ h^q
\]
on $S\simeq [b',b]\times \IR$, for every $q\in \IZ$.

Then
\[
t^q\circ s\circ f= t^q\circ f' = f'\circ h^q = s \circ f\circ h^q=s\circ t^q\circ f. 
\]

Choose a point $x\in \ran f$. The relation shows that for every $q\in \ZI$ the commutator
\[
[t^q,s]\in \Stab_\G(x),
\]
which is a finite group by assumption. Therefore we can find $q\neq q'$ such that 
\[
[t^q,s]=[t^{q'},s]
\]
so 
\[
[t^{q-q'},s]=e.
\]
Since 
\[
t^q\circ f =  f\circ h^q
\]
for every $q$, it follows that $t$ has infinitely order in $\G$.

Assume towards a contradiction that $s$ has finite order in $\G$. Then the subset of $X$ 
\[
\bigcup_{k=0}^{\text{order of }s} \ran s^k\circ (f\vee f')
\]
is a simplicial cylinder in $X$, whose existence is precluded  by assumption. 

Thus, $\langle t^{q-q'},s\rangle\simeq \ZI^2$ in $\G$.
\end{proof}

\begin{remark} The space of marked flats $\MF(X)$ is the set
\[
\MF(X)=\{\IR^2\inj X\}
\]
of all isometric embeddings $\IR^2\inj X$.
This is a classical object associated with $X$; it has been studied in particular by Pansu and is useful when studying the rigidity of symmetric spaces (cf.\ \cite[\textsection 14]{mostow}). 

The space $\MF(X)$ has two commuting actions of $\IR^2$ and $\Gamma$ 
\[
t\cdot f:= f\circ t^{-1} \text{ and } s\cdot f:= s\circ f
\]
respectively at the source and the range, for $t$ a translation of $\IR^2$, $s\in\G$, and $f\colon \IR^2\inj X$ an embedding.

If $X$ is locally finite and $X/\G$ is compact then $\MF(X)/\G$ is a compact space. It is not empty if $X$ is not hyperbolic, and it is endowed with a structure of a lamination whose leaves are associated with flat planes in $X$ and given by the action of  $\IR^2$.

If $X$ is a ring complex of type $\Aut(F_2)$ we have an obvious map
\[
\MF(X)\to \MP({\Aut(F_2)})
\]
which is continuous and open. 

Note that in the case of ring complexes of type $\Aut(F_2)$, the subset $\MF_x(X)$ corresponding to  flats in $X$ of type $x=$ a), b), $\ldots$ is closed whenever flats of type $x$ are geometrically periodic (i.e.\ the simplicial isometry group acts with compact quotient). This holds for flats of type a) or b), for example. 

A ring complex $X$ is said to be \emph{full} if the map $\MF(X)\to \MP(T)$ is surjective with respect to $T=T(X)$, where $T(X)$ denotes the type of $X$ (minimal sets of shapes and rings in $X$). 

This notion of ``fullness'' can be thought of as a weakening of the idea of ``space of maximal rank''  from rank interpolation.

For example the Brady complex is full, and we think of it as a space of maximal rank among ring complexes of type $\Aut(F_2)$. It plays the role of Euclidean buildings in \cite{random}, which are the prototypical ``spaces of maximal rank'' (in their respective classes, namely among spaces of type $\widetilde A_2$, of type $\widetilde B_2$ or of type $\widetilde G_2$). 

Observe also that the classification of Euclidean buildings into types is primarily a classification of their  apartments (flats) and associated Coxeter groups, which in spirit is not very far from the definition of ring complexes being ``of type $\Aut(F_2)$'' described in the present paper.       
 \end{remark}

\section{Randomizing $\Aut(F_2)$}\label{S - random autf2}

Let $\G_p\leq \G_0:=\Aut(F_2)$ be finite index subgroups and consider the deterministic data $(\G_0,\{\G_p\}, X_0)$ in the sense of \cite{random}, where $X_0$ is the Brady space.

We only discuss the density model. Results can also be established in the bounded model (for example). 

Recall that if $\IP_p$ denotes a process for selecting finite random subsets of elements in the set 
\[
\sC_p:=\{\G_p\text{-orbits of faces in }X_0\}
\]
then the density model can be defined in the following way, given a  parameter $\delta\in (0,1)$.

\begin{definition}  The \emph{density model} over $(\G_0,\{\G_p\},X_0)$ is  associated with
\begin{center}
 $\IP_p:=$ ``choose  $|\sC_p|^\delta$ chambers in $\sC_p$, uniformly and independently at random''. 
\end{center}
\end{definition}

We have the following  phase transition at $\delta=\frac 2 3$.

\begin{proposition}[See {\cite[\textsection5]{random}}]
If  $\G$ denotes the  random group in the density model over $(\G_0,\{\G_p\},X_0)$ then
\begin{enumerate}
\item  if $\delta<2/3$ then $\G$  acts properly on a CAT(0) space $X$ without boundary such that $X/\G$ compact
\item  if $\delta>2/3$ then $\G$ splits off a free factor isomorphic to a free group $F_n$ on $n$ generators.
\end{enumerate}
\end{proposition}

\begin{lemma}
The random group in the density model over $(\G_0,\{\G_p\},X_0)$ acts on a ring complex of type $T$.
\end{lemma}

\begin{proof}
By construction the random group $\G$ at  arbitrary density acts on a space $X$ which is locally isometric to a subspace of $X_0$ and therefore the links of $X$ embed $L$. Note that if $1/2<\delta<2/3$ the contraction of free faces removes free paths in the link, and if $\delta>2/3$ the links are not necessarily connected. The shape set of $X$ is also a subset of the shape set of $X$, so $X$ is of type $T$. This  is true  more generally of the $\alpha$-model \cite[\textsection 4]{random}.
\end{proof}

The following result is an application of Theorem \ref{Th - gromov Z2 puzzles}.

\begin{theorem}
For arbitrary deterministic data $(\G_0,\{\G_p\},X_0)$ with $\G_0=\Aut(F_2)$, $\G_p$ are finite index subgroups, and $X_0$ denotes the brady complex, the random group in the density model over $(\G_0,\{\G_p\},X_0)$ satisfies the flat closing conjecture. 
\end{theorem}

This may be compared to the randomized lattices in nonarchimedean Lie groups \cite{random}, for which the corresponding statement is open. %The flat closing conjecture for spaces of intermediate rank becomes fuzzy as the ``intermediate rank'' is \emph{moderately} close to $2$. 

\begin{proof}
Let $\G\acts X$ denote the action of the random group on its associated random classifying space. Since $X$ is nonpositively curved \cite{random} it is acylindrical. In order to apply Theorem \ref{Th - gromov Z2 puzzles} we have to check the unique extension property for $X$. This follows from a generalization of \cite[Lemma 11]{random} to strips, which are simply connected (see also  \cite[Lemma 13]{rd}.). Namely if $A\subset \sC_p$ is an arbitrary subset and, in the notation of \cite[Lemma 11]{random}, if
\[
S\subset X_A:=X_0- \bigcup_{C\in A}\stackrel \circ C
\] 
is a strip,  then the map $j:S\to \Pi\subset X_A$ lifts to unique map  $\tilde j :   S\to\tilde \Pi\subset \tilde X_A$ through any point $\tilde j(s)=\tilde x\in \tilde X_A$, $s\in S$ such that $\pi_A(\tilde x)=j(s)=:x$ and we have a commutative diagram
\[
\xymatrix{
 &\tilde X_A\ar[d]^{\pi_A}&\\
S\ar[r]_{j}\ar[ur]^{\tilde j}  & X_A}
\]   

We prove that the strip $S$ is uniquely $L$-embeddable in $X$ if it is so in  the Brady complex $X_0$. Since both the  $\diamond$-strip and the $\triangle$-strip are uniquely embeddable in $X_0$ (which is an easy exercise), the result will follow.  

Consider two simplicial embeddings $f,f'\colon S\inj \tilde X_A$ and a translation $g\colon S\to S$ 
\[
f'=f\circ g \text{ on } L. 
\]
We have to show that $f'=f\circ g$  on $S$. 

Let $j:=\pi_A\circ f\circ g$ and $j':=\pi_A\circ f'$. These are two embeddings $S\to X_A\subset X_0$ which coincide on $L$. By the the unique strip extension property for $S$ in $X_0$, we have that $j=j'$ on $S$. 

Let $s\in L$ and $\tilde x= f'(s)=f\circ g(s) \in f'(L)=f\circ g(L)$. There is a unique lift $\tilde j$ of $j$ such that $\tilde j(s)=\tilde x$ and a unique lift  $\tilde j'$ of $j'$ such that $\tilde j'(s)=\tilde x$ so  $\tilde j=\tilde j'$ on $S$ (since $j=j'$ on $S$). Applying uniqueness again $f\circ g=\tilde j$ on $S$ and $f'=\tilde j'$ on $S$. So $f'=f\circ g$  on $S$. 

This establishes the unique extension property for $\tilde X_A$ with respect to an arbitrary --- in particular for a random --- subset $A\subset \sC_p$.
\end{proof}

The $\Aut(F_2)$ models differ from the nonarchimedean lattices models of \cite{random} in many respects. 
  What we prove below is  that for suitable choices of  the deterministic data, the presence of periodic flats persists at \emph{arbitrary} densities $\delta\in (0,1)$. 
  This relies on the fact that  the following ``small tori property'' can be achieved for $\Aut(F_2)$ models, for example for deterministic data associated with the symmetric group, using superexponential growth.

\begin{definition}\label{D - small tori property} 
We say that the deterministic data $(\G_0,\{\G_p\},X_0)$  has the \emph{small tori property} (in the density model) if for any $p\geq 1$ there exists an embedding $\ZI^2\inj \G_0$ and a simplicial embedding $\IR^2\inj X_0$ with $\ZI^2\acts\IR^2$ (freely and cocompactly) such that
\[
\forall \delta>0, \ |(\IR^2)^{(2)}/(\ZI^2\cap \G_p)|\leq O(|X^{(2)}/\G_p|^\delta).
\]
where $^{(2)}$ indicates  2-faces.\end{definition}

This is easily seen to be an algebraic condition on the growth of the finite index subgroups relative to the embedding of $\ZI^2$, at least if $X$ is a CAT(0) space (see the proof of Theorem \ref{T - arbitrary density}).

The persistence of periodic flat planes at arbitrary density  will follow from:

\begin{lemma}\label{L - small tori arbitrary density}
If $(\G_0,\{\G_p\},X_0)$ has the small tori property, where $X_0$ is a CAT(0) space, then $\ZI^2$ embeds in the random group at arbitrary density $\delta\in (0,1)$ with overwhelming probability.
\end{lemma}

The proof is  extracted from \cite[\textsection 6]{random}.

\begin{proof}
 Let $\delta\in (0,1)$ be given. Let $\sC_p$ be the set of $\G_p$-orbits of 2-cells in $X_0$ (which can be triangles or lozenges)
and let $F_p$ be the set of  $\ZI^2\cap \G_p$ orbits of 2-cells in $\IR^2$ viewed as a flat plane in $X_0$. 

Let $\alpha<1-\delta$. By the small tori property we have
\[
 \ |F_p|\leq O(|X^{(2)}/\G_p|^{\alpha}).
\]
Write $F_p'$ for the image of $F_p$ in $\sC_p$ so
\[
 \ |F_p'|\leq C |\sC_p|^{\alpha}
\]
for some constant $C>0$.

Consider a sequence $Y_1,Y_2,Y_3,\ldots$ of i.i.d.\ random variables with values in $\sC_p$.
The probability $\IP(E_p)$ of the event:

\[
E_p= \{Y_i\notin F_p', \; \forall i\leq |\sC_p|^\delta\}
\]
satisfies
\[
\IP(E_p)\geq \left(1-{|F_p'|\over |\sC_p|}\right)^{|\sC_p|^\delta}
\geq  e^{-2|\sC_p|^{\delta-1}|F_p'|}\geq e^{-2C|\sC_p|^{\delta-1+\alpha}}
\] 
so $\IP(E_p)\to 1$. Applying Lemma 11 in \cite{random} we have that $\ZI^2\inj \G$ with overwhelming probability at density $\delta$.
\end{proof}

This can be applied to the case of $\Aut(F_2)$ in the following way.
Write $F_2=\langle x, y\rangle$ and define automorphisms $\alpha,\beta\in \Aut(F_2)$  
\[
\alpha\colon x\to x,\  y\to x^2y,\ \beta\colon x\to xy^2,\ y\to y
\]
so
\[
\bar \alpha=\matr 1  2  0  1 \text{ and } \bar \beta=\matr 1 0 2 1
\]
where $\bar{\alpha}$, $\bar{\beta}$ are the reduction  relative to the short exact sequence  
\[
\mathrm{IA}_2\inj  \Aut(F_2)\stackrel{\bar{}}{\surj} \GL_2(\ZI).
\]
 $\mathrm{IA}_2= \Inn(F_2)=\langle \alpha_0,\beta_0\rangle \simeq F_2$ is the subgroup of inner automorphisms:
 \[
\alpha_0\colon x\to x,\  y\to xyx^{-1},\ \beta_0\colon x\to yxy^{-1},\ y\to y.
\]
 It is well--known that $\langle\bar \alpha, \bar\beta\rangle\simeq F_2$ so we have a split exact sequence
\[
\mathrm{IA}_2\inj \mathrm{IA}_2\rtimes \langle \alpha,\beta\rangle\surj \langle \bar \alpha,\bar \beta\rangle.
\]

\begin{theorem}[Periodic flat planes at arbitrary density] \label{T - arbitrary density}
Let $A_p\normal F_2$ be a characteristic subgroup of $F_2$ and let $B_p \normal F_2$ be a normal subgroup such that $F_2/B_p$ is the  group $S_p$ of permutations of the set with $p$ elements. Suppose that $[F_2:A_p]\leq e^{O(p)}$. If the deterministic data $(\G_0,\{\G_p\},X_0)$ satisfy: 
\begin{itemize}
\item $\G_0=\Aut(F_2)$
\item $X_0$ is the Brady complex
\item $\G_1=\mathrm{IA}_2\rtimes \langle \alpha,\beta\rangle$ (a subgroup of index 24 in $\G_0$)
\item for $p\geq 2$, we have $\G_p:=A_p\rtimes B_p$ where we view $A_p$ as a subgroup of $\mathrm{IA}_2$ and $B_p$ as a subgroup of $\langle \bar \alpha,\bar \beta\rangle$.
\end{itemize}
Then  with overwhelming probability $\ZI^2\inj \G:=$ the random group at density $\delta\in (0,1)$.
\end{theorem}
 
\begin{proof}
Since $X_0$ is a CAT(0) space the triple $(\G_0,\{\G_p\},X_0)$ has the small tori property if and only if there exists an embedding $\ZI^2\inj \G$ such that 
\[
\forall \delta>0, \ |\ZI^2/(\ZI^2\cap \G_p)|\leq O(|\G/\G_p|^\delta).
\]
We have 
\[
|\langle \alpha_0\rangle/(\langle \alpha_0\rangle\cap A_p)|\leq e^{O(p)}, \text{ and similarly } |\langle \beta_0\rangle/(\langle \beta_0\rangle\cap A_p)|\leq e^{O(p)}
\]
by assumption and
\[
|\langle \alpha\rangle/(\langle \alpha\rangle\cap B_p)|\leq g(p), \text{ and similarly } |\langle \beta\rangle/(\langle \beta\rangle\cap B_p)|\leq g(p)
\]
where 
\[
g(p):=\text{ maximal order of an element in }S_p 
\]
is the Landau function (see Wikipedia). Note that the value $g(p)$ can be indeed achieved in our situation \cite[\textsection 2.63]{Dixon}. 

Landau has proved that 
\[
\lim_{p\to\infty} \frac{\ln g(p)}{\sqrt{p\ln p}}=1
\]
therefore
\[
|\langle \alpha\rangle/(\langle \alpha\rangle\cap B_p)|\leq e^{O(p)}, \text{ and similarly } |\langle \beta\rangle/(\langle \beta\rangle\cap B_p)|\leq e^{O(p)}.
\]

If $\Lambda\simeq \ZI^2$ denotes the subgroup of $\Aut(F_2)$ generated either by $\alpha_0,\alpha$ or by $\beta_0,\beta$ then we have 
\[
|\Lambda/(\Lambda\cap \G_p)|\leq e^{O(p)}.
\]
On the other hand by Stirling's approximation 
\[
|\G/\G_p|\geq p!\geq \left( \frac p e\right )^p. 
\]
This shows that
\[
\forall \delta>0, \ |\Lambda/(\Lambda\cap \G_p)|\leq O(|\G/\G_p|^\delta).
\]
Then Lemma \ref{L - small tori arbitrary density} applies. 
\end{proof}

\begin{remark} 1) $\Aut(F_2)$ has many interesting finite quotients leading to interesting phase transitions in the corresponding density model. We hope to come back to this elsewhere. 

2) The variation of the small tori property which is implicit in \cite{random} is the assertion that there is $\alpha\geq 0$ such that for any $p\geq 1$ there are $\ZI^2\inj \G_0$ and $\IR^2\inj X_0$ with $\ZI^2\acts\IR^2$ freely and cocompactly such that
\[
\forall \delta>\alpha, \ |(\IR^2)^{(2)}/(\ZI^2\cap \G_p)|\leq O(|X^{(2)}/\G_p|^\delta).
\] 
\end{remark}

 \section{The explicit construction}\label{S - explicit examples 2}

The construction provides a non positively curved simplicial complex $X$ and a group $\G$ acting freely on $X$ with 3 orbits of vertices. 

The link at every vertex is a ``double'' of the link of the Brady complex as explicited on the following drawing (in red and black). See \cite[Fig.\ 6]{CP} for the link of the Brady complex itself. 

\[
\includegraphics[width=10cm]{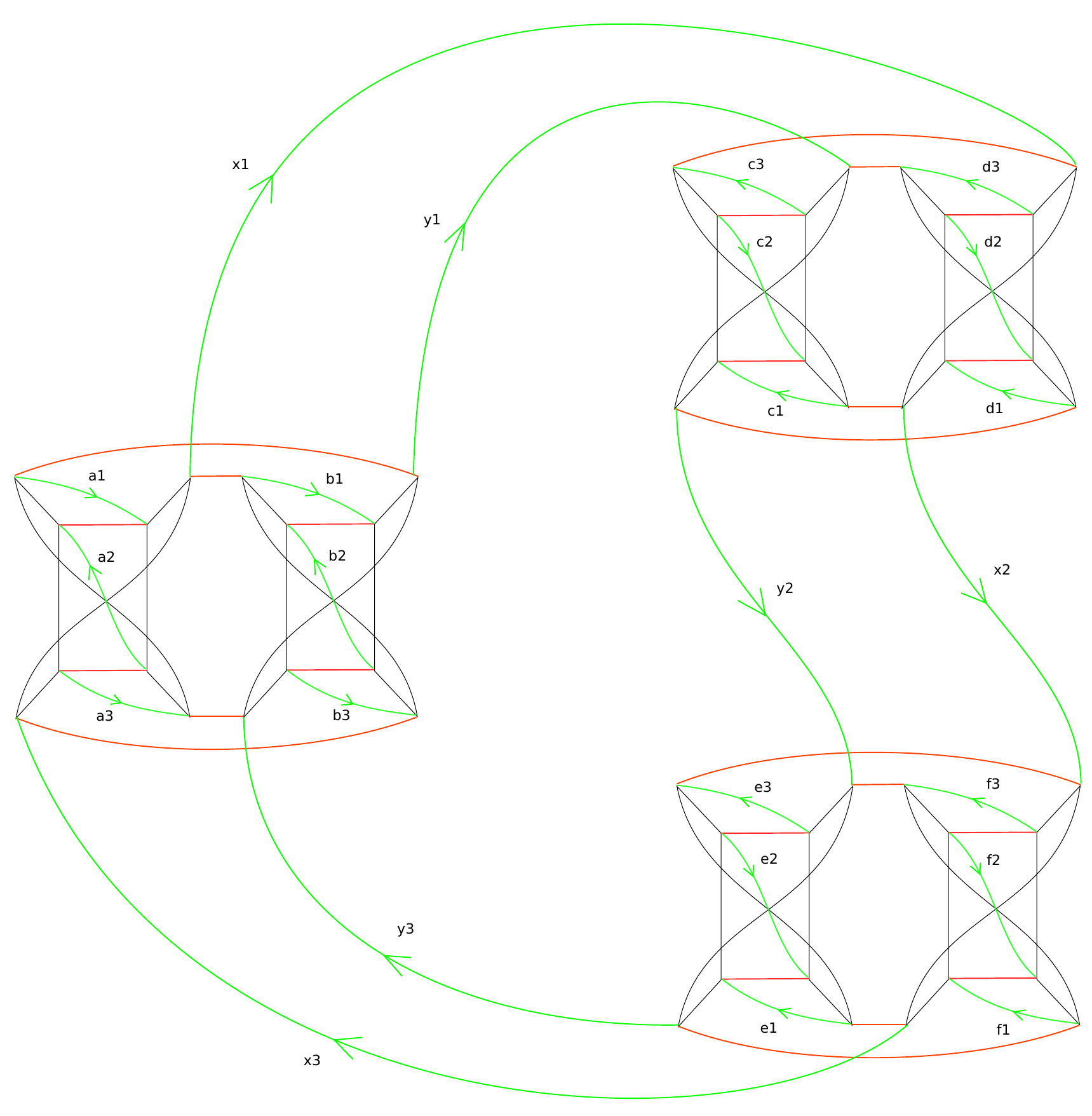}
\]

The figure shows $X/\G$. The 1-skeleton (in green) of $X/\G$ 
supports the following set of faces:

\begin{itemize}
\item[] triangles: $\bullet_1\bullet_2\bullet_3$ where $\bullet=a,b,c,d,e,f,x,y$
\item[] lozenges: $\bullet_1\bullet_2'\bullet_3'\bullet_2$, where $\bullet=a,b,c,d,e,f$ and $x_1d_3y_1'b_1'$, $x_1c_3'y_1'a_1$, $x_2f_3y_2'c_1'$, $x_2e_3'y_2'd_1$, $x_3a_3y_3'e_1'$, $x_3b_3'y_3'f_1$.
\end{itemize}

It is easy to see that the  complex $X$ thus constructed is a ring complex of type $\Aut(F_2)$. Further inspection shows that $X$ fails to be hyperbolic and has the isolated flats property. Every link contains 8 cycles of length $2\pi$ and the corresponding puzzles are of type A.  

This group appears to be an intriguing sort of ``double'' of $\Aut(F_2)$.

\end{document}